\documentclass{amsart}
\usepackage{amsmath}
\usepackage{amssymb}
\usepackage{amsfonts}

\newtheorem{theorem}{Theorem}
\theoremstyle{plain}

\newtheorem{corollary}{Corollary}

\newtheorem{definition}{Definition}

\newtheorem{remark}{Remark}

\numberwithin{equation}{section}
\input{tcilatex}

\begin{document}
\title[Introduction to Differential Geometry... ...]{Introduction to
Differential Geometry in the Three-Dimensional Field of Complex Vectors}
\author{Branko Sari\'{c}}
\address{Faculty of Technical Sciences, University of Kragujevac, \v{C}a\v{c}%
ak 32000, Serbia}
\email{saric.b@mts.rs}
\date{July 29, 2026}

\begin{abstract}
Building on a commutative geometric algebra developed for 2$\mathbb{D}$ and 3%
$\mathbb{D}$ fields of complex vectors, this paper develops a vector
framework for the differential geometry of compact \textit{Hermitian}
manifolds embedded in an ambient three-dimensional field of complex vectors.
The underlying algebraic framework is constructed from the commutative
geometric product of complex vectors, yielding a unified commutative algebra
in both the 2$\mathbb{D}$ and 3$\mathbb{D}$ settings together with the
corresponding vector differential and integral identities.

These algebraic results provide the intrinsic foundation for a differential
geometric formulation that is developed independently of the classical
scalar coordinate approach. Within this framework, compact \textit{Hermitian}
manifolds are described entirely in vector form. Fundamental tensor
relations governing the intrinsic and extrinsic geometry of complex
manifolds are derived in vector form. The resulting framework establishes a
unified algebraic and geometric setting in which the differential geometry
of complex manifolds is developed directly from the commutative algebra of
complex vectors.
\end{abstract}

\maketitle

\section{Introduction}

Complex numbers admit a natural geometric interpretation in the \textit{%
Euclidean} plane $%
\mathbb{R}
^{2}$, where each complex number $z\in 
\mathbb{C}
$ is identified by an ordered pair of real numbers. While addition coincides
with ordinary vector addition, multiplication is traditionally defined as an
algebraic operation on complex numbers rather than as an intrinsic operation
on vectors. Consequently, classical vector algebra provides no commutative
vector multiplication that reproduces complex multiplication while remaining
entirely within the vector space.

This paper is based on the observation that complex multiplication admits an
intrinsic geometric realization on 2$\mathbb{D}$ complex vectors. This leads
to the definition of a commutative geometric product that endows the space
of 2$\mathbb{D}$ complex vectors with a field structure algebraically
isomorphic to the classical field $%
\mathbb{C}
$. The same construction extends naturally to a 3$\mathbb{D}$ field of
complex vectors, where it induces a commutative algebraic structure based on
the geometric product, together with the corresponding vector differential
and integral identities.

The purpose of this paper is not to reformulate classical complex analysis
or differential geometry in a different notation, but to develop an
intrinsic vector framework for differential geometry based on the
commutative algebra of complex vectors. Within this framework, vector
operations replace the traditional scalar coordinate formalism while
preserving the underlying geometric content.

The first three sections of the paper establish the required algebraic
foundations. So, \textit{Section} 2 constructs the 2$\mathbb{D}$ field of
complex vectors, and \textit{Section} 3 extends this construction to a 3$%
\mathbb{D}$ field of complex vectors $\mathbf{V}_{\mathbb{C}}^{3}$ and
develops the corresponding differential and integral identities. Building
upon these results, \textit{Section} 4 develops the differential geometry of
compact \textit{Hermitian} manifolds embedded in an ambient 3$\mathbb{D}$
field of complex vectors. Fundamental vector and tensor identities in the
final section are derived entirely in the corresponding vector form,
including a vector formulation of the \textit{Gauss} equation together with
the associated differential identities governing the intrinsic and extrinsic
geometry of \textit{Hermitian} manifolds in $\mathbf{V}_{\mathbb{C}}^{3}$.
The resulting theory provides a unified geometric framework in which
differential geometry itself arises naturally from the commutative geometric
algebra of complex vectors.

\section{The 2$\mathbb{D}$ field of complex vectors}

Unlike the classical identification of the complex plane $%
\mathbb{C}
$ with the \textit{Euclidean} plane $%
\mathbb{R}
^{2}$, the algebraic structure developed in this section is defined
intrinsically on vectors through geometric operations. Specifically, we
construct a commutative geometric product of complex vectors that reproduces
complex multiplication while remaining entirely within the vector space.

Let $i%
\mathbb{R}
=\{ix\mid x\in 
\mathbb{R}
\}$ denote 1$\mathbb{D}$ real vector space of purely imaginary numbers.
Then, the \textit{Cartesian} product $V_{%
\mathbb{C}
}=%
\mathbb{R}
\times i%
\mathbb{R}
$ is a 2$\mathbb{D}$ real vector space with basis vectors $(1,i0)$ and $%
(0,i) $. Its elements are ordered pairs $(x,iy)$, which are naturally
identified with complex numbers $z\in 
\mathbb{C}
$ through the bijection $\Phi :$ $%
\mathbb{R}
\times i%
\mathbb{R}
\rightarrow $ $%
\mathbb{C}
$, $\Phi (x,iy)=z$. To endow $V_{%
\mathbb{C}
}$ with the algebraic structure corresponding to complex multiplication, we
associate each element $(x,iy)$ with the symmetric matrix $\left[ 
\begin{array}{cc}
x & iy \\ 
iy & x%
\end{array}%
\right] $. Thus, matrix multiplication of two elements $\left( a,ib\right) $
and $\left( c,id\right) $ 
\begin{equation*}
\left( a,ib\right) \left( c,id\right) \rightleftharpoons \left[ 
\begin{array}{cc}
a & ib \\ 
ib & a%
\end{array}%
\right] \left[ 
\begin{array}{cc}
c & id \\ 
id & c%
\end{array}%
\right] =
\end{equation*}%
\begin{equation}
=\left[ 
\begin{array}{cc}
ac-bd & i(ad+bc) \\ 
i(ad+bc) & ac-bd%
\end{array}%
\right] \leftrightharpoons (ac-bd,i(ad+bc))\text{,}  \label{1}
\end{equation}%
induces a commutative multiplication on $V_{%
\mathbb{C}
}$, identical to the multiplication of two complex numbers. Every nonzero
element $(x,iy)\in V_{\mathbb{%
\mathbb{C}
}}^{\ast }=V_{\mathbb{%
\mathbb{C}
}}\backslash \{(0,i0)\}$ admits the inverse $(x,iy)^{-1}=\left( x,-iy\right)
/(x^{2}+y^{2})$, corresponding exactly to the inverse of the associated
matrix. Consequently, multiplication in $V_{%
\mathbb{C}
}$ is associative, commutative, distributive over addition, and every
nonzero element is invertible. The purpose of introducing $V_{%
\mathbb{C}
}$ is to establish the ordered-pair framework underlying the 2$\mathbb{D}$
field of complex vectors developed below. In analogy with $V_{%
\mathbb{C}
}$, whose elements are ordered pairs consisting of a real number and a
purely imaginary number, we shall represent the elements of the 2$\mathbb{D}$
field of complex vectors as ordered pairs consisting of a real vector and a
purely imaginary vector. Thus, the 2$\mathbb{D}$ field of complex vectors
provides a geometric realization of the algebraic structure of complex
numbers, in which scalar components are replaced by vector counterparts,
while preserving the corresponding algebraic operations.

Let $\mathbf{1}_{x}$ and $\mathbf{1}_{y}$ denote the unit coordinate vectors
of the \textit{Cartesian} plane $V_{2}$, and let $\mathbf{V}_{\mathbb{C}}${}
denote the geometric realization of the 2$\mathbb{D}$ vector space $V_{%
\mathbb{%
\mathbb{C}
}}$. It is equipped with an orthonormal basis ($\mathbf{e,\hat{e}}$), while
the basis vectors $\mathbf{e}$ and $\mathbf{\hat{e}}$ are represented by the
ordered pairs $(\mathbf{1}_{x},i\mathbf{0}_{y})$ and $(\mathbf{0}_{x},i%
\mathbf{1}_{y})$, respectively. Naturally, the basis vectors $\mathbf{e}$
and $\mathbf{\hat{e}}$ correspond to the basis vectors $(1,i0)$ and $(0,i)$
of $V_{\mathbb{%
\mathbb{C}
}}$, respectively. Every complex vector $\mathbf{\varrho }\in \mathbf{V}_{%
\mathbb{C}}$ is written in the form $\mathbf{\varrho }=x\mathbf{e}+y\mathbf{%
\hat{e}}$, where $x$ and $y\ $are real numbers. It represents the ordered
pair consisting of the real vector $x\mathbf{1}_{x}$ and the purely
imaginary vector $iy\mathbf{1}_{y}$. If $\mathbf{\bar{\varrho}}=x\mathbf{e}-y%
\mathbf{\hat{e}}$ is the conjugate of a complex vector $\mathbf{\varrho }\in 
\mathbf{V}_{\mathbb{C}}$, then the standard dot and cross products of $%
\mathbf{\bar{\varrho}}_{1}$ and $\mathbf{\varrho }_{2}$ in $\mathbf{V}_{%
\mathbb{C}}$ are defined by%
\begin{equation}
\mathbf{\bar{\varrho}}_{1}\cdot \mathbf{\varrho }_{2}=(x_{1}\mathbf{e}-y_{1}%
\mathbf{\hat{e})\cdot (}x_{2}\mathbf{e}+y_{2}\mathbf{\hat{e}}%
)=x_{1}x_{2}+y_{1}y_{2}\text{ and}  \label{2}
\end{equation}%
\begin{equation*}
\mathbf{\bar{\varrho}}_{1}\times \mathbf{\varrho }_{2}=\hat{n}\left\vert 
\begin{array}{ccc}
\mathbf{e} & \mathbf{\hat{e}} & \mathbf{n} \\ 
x_{1} & -y_{1} & 0 \\ 
x_{2} & y_{2} & 0%
\end{array}%
\right\vert =(x_{1}y_{2}+x_{2}y_{1})\mathbf{n}\text{,}
\end{equation*}%
where $\mathbf{n=e}\times \mathbf{\hat{e}}$ is the unit normal vector to the
plane $\mathbf{V}_{\mathbb{C}}$, such that $\mathbf{n\times e}=\mathbf{\hat{e%
}}$ and $\mathbf{n\times \hat{e}}=\mathbf{e}$.

Motivated by the multiplication of complex numbers and (\ref{1}), we now
introduce an intrinsic commutative geometric product of complex vectors, 
\cite{Sar1}.

\begin{definition}
Let $\mathbf{V}_{\mathbb{C}}$ be equipped with two binary operations $\circ :%
\mathbf{V}_{\mathbb{C}}\times \mathbf{V}_{\mathbb{C}}\rightarrow \mathbf{V}_{%
\mathbb{C}}$ and $\wedge :\mathbf{V}_{\mathbb{C}}\times \mathbf{V}_{\mathbb{C%
}}\rightarrow \mathbf{V}_{\mathbb{C}}$ called the inner vector product and
the outer vector product, respectively. The geometric vector product $\star :%
\mathbf{V}_{\mathbb{C}}\times \mathbf{V}_{\mathbb{C}}\rightarrow \mathbf{V}_{%
\mathbb{C}}$ is defined by%
\begin{equation}
\mathbf{\bar{\varrho}}_{1}\star \mathbf{\varrho }_{2}=\mathbf{\varrho }%
_{1}\circ \mathbf{\varrho }_{2}+\mathbf{\varrho }_{1}\wedge \mathbf{\varrho }%
_{2}\text{.}  \label{3}
\end{equation}
\end{definition}

The algebraic properties of ($\mathbf{V}_{\mathbb{C}},\mathbf{+},\star $)
follow directly from the definitions of the inner and outer vector products.
Specifically, if 
\begin{equation}
\mathbf{\varrho }_{1}\circ \mathbf{\varrho }_{2}=(\mathbf{\bar{\varrho}}%
_{1}\cdot \mathbf{\varrho }_{2})\mathbf{e}\text{ and }\mathbf{\varrho }%
_{1}\wedge \mathbf{\varrho }_{2}=(\mathbf{\varrho }_{1}\times \mathbf{%
\varrho }_{2})\times \mathbf{e}\text{,}  \label{4}
\end{equation}%
then the geometric vector product $\mathbf{\bar{\varrho}}_{1}\star \mathbf{%
\varrho }_{2}$, or simply $\mathbf{\bar{\varrho}}_{1}\mathbf{\varrho }_{2}$,
is commutative.

Furthermore, $\mathbf{\bar{\varrho}\varrho }=\varrho ^{2}\mathbf{e}$ and $%
\mathbf{\bar{\varrho}}=\varrho ^{2}\mathbf{\varrho }^{-1}$, where $\varrho
=\left\Vert \mathbf{\varrho }\right\Vert =\left\Vert \left( x,iy\right)
\right\Vert =\left\vert \sqrt[2]{x^{2}+y^{2}}\right\vert $ is the \textit{%
Hermitian} norm on $\mathbf{V}_{\mathbb{C}}$ induced by the geometric
product $\mathbf{\bar{\varrho}\varrho }$. Consequently, every nonzero vector
in $\mathbf{V}_{\mathbb{C}}$ admits a multiplicative inverse, so ($\mathbf{V}%
_{\mathbb{C}},\mathbf{+},\star $) forms a commutative field with the vector $%
\mathbf{e}$ as the multiplicative unit. The inverse vector $\mathbf{\varrho }%
^{-1}$ therefore enables division by nonzero vectors, and the resulting
field of complex vectors $\mathbf{V}_{\mathbb{C}}$ is algebraically
isomorphic to the classical complex field $%
\mathbb{C}
$.

\section{The 3$\mathbb{D}$ field of complex vectors}

Let $\mathbf{1}_{x}$, $\mathbf{1}_{y}$ and $\mathbf{1}_{z}$ form an
orthonormal basis for a 3$\mathbb{D}$ \textit{Euclidean} vector space $%
\mathbf{V}^{3}$. The 3$\mathbb{D}$ space of complex vectors is constructed
as the direct sum of three mutually orthogonal copies of the 2$\mathbb{D}$
field of complex vectors $\left\langle \mathbf{V}_{\mathbb{C}}\right\rangle
_{\alpha }$ ($\alpha =1,2,3$), each equipped with its own orthonormal basis,
so that: $\mathbf{e}_{1}=\left( \mathbf{1}_{x},i\mathbf{0}_{y}\right) $ and $%
\mathbf{\hat{e}}_{1}=\left( \mathbf{0}_{x},i\mathbf{1}_{y}\right) $, $%
\mathbf{e}_{2}=\left( \mathbf{1}_{y},i\mathbf{0}_{z},\right) $ and $\mathbf{%
\hat{e}}_{2}=\left( \mathbf{0}_{y},i\mathbf{1}_{z}\right) $, as well as $%
\mathbf{e}_{3}=\left( \mathbf{1}_{z},i\mathbf{0}_{x}\right) $ and $\mathbf{%
\hat{e}}_{3}=\left( \mathbf{0}_{z},i\mathbf{1}_{x}\right) $, are orthonormal
bases of $\left\langle \mathbf{V}_{\mathbb{C}}\right\rangle _{1}$, $%
\left\langle \mathbf{V}_{\mathbb{C}}\right\rangle _{2}$ and $\left\langle 
\mathbf{V}_{\mathbb{C}}\right\rangle _{3}$, respectively. The complex vector%
\begin{equation}
\mathbf{\varrho }=x_{\alpha }\mathbf{e}^{\alpha }+\hat{x}_{\alpha }\mathbf{%
\hat{e}}^{\alpha }=\mathbf{\varrho }_{\alpha }\mathbf{e}^{\alpha }\text{,}
\label{5}
\end{equation}%
where $\mathbf{\varrho }_{\alpha }=\left\langle x\mathbf{e}+\hat{x}\mathbf{%
\hat{e}}\right\rangle _{\alpha }$ denotes the component complex vector
belonging to $\left\langle \mathbf{V}_{\mathbb{C}}\right\rangle _{\alpha }$,
is a vector of the 3$\mathbb{D}$ space of complex vectors. Throughout this
paper, the \textit{Einstein} summation convention is adopted, so that
repeated upper and lower indices are summed over the range of the index. The
conjugate of a complex vector $\mathbf{a=a}_{\alpha }\mathbf{e}^{\alpha }$
is defined by $\mathbf{\bar{a}}=\mathbf{\bar{a}}_{\alpha }\mathbf{e}^{\alpha
}$, where $\mathbf{a}_{\alpha }=\langle a\mathbf{e}+\hat{a}\mathbf{\hat{e}%
\rangle }_{\alpha }$ and $\mathbf{\bar{a}}_{\alpha }=\langle a\mathbf{e}-%
\hat{a}\mathbf{\hat{e}\rangle }_{\alpha }$.

\begin{definition}
The geometric product of two $3\mathbb{D}$ complex vectors $\mathbf{\bar{a}}$
and $\mathbf{b}$ is defined componentwise by the geometric products of the
corresponding component vectors $\mathbf{\bar{a}}_{\alpha }$ and $\mathbf{b}%
_{\alpha }$, belonging to the mutually orthogonal component fields $%
\left\langle \mathbf{V}_{\mathbb{C}}\right\rangle _{\alpha }$,
\end{definition}

\begin{equation}
\mathbf{\bar{a}b}=\mathbf{a\circ b}+\mathbf{a\wedge b}=\langle \mathbf{%
a\circ b+a\wedge b\rangle }_{\alpha }\mathbf{e}^{\alpha }=\langle \mathbf{%
\bar{a}b\rangle }_{\alpha }\mathbf{e}^{\alpha }\text{.}  \label{6}
\end{equation}

For every nonzero complex vector $\mathbf{a=a}_{\alpha }\mathbf{e}^{\alpha }$%
, we define%
\begin{equation}
\mathbf{a}^{-1}=\mathbf{a}_{\alpha }^{-1}\mathbf{e}^{\alpha }\text{ ,}
\label{7}
\end{equation}%
where 
\begin{equation}
\mathbf{a}_{\alpha }^{-1}=\langle \mathbf{\bar{a}}/\left\Vert \mathbf{a}%
\right\Vert ^{2}\rangle _{\alpha }.  \label{8}
\end{equation}

\begin{theorem}[Fundamental algebraic properties]
Let $\mathbf{V}_{\mathbb{C}}^{3}$ be the $3\mathbb{D}$ space of complex
vectors equipped with the geometric product defined by $($\ref{6}$)$. Then,

$1$. the geometric product is bilinear,

$2$. the geometric product is commutative,

$3$. the geometric product is associative,

$4$. the complex vector $(\mathbf{1}_{\mathbf{V}},i\mathbf{0})\in \mathbf{V}%
_{\mathbb{C}}^{3}$, where $\mathbf{1}_{\mathbf{V}}=\mathbf{e}_{\alpha }%
\mathbf{e}^{\alpha }$, is the multiplicative identity of the geometric
product, satisfying $\mathbf{1}_{\mathbf{V}}\mathbf{a}=\mathbf{a1}_{\mathbf{V%
}}=\mathbf{a}$,

$5$. every nonzero vector $\mathbf{a}\in \mathbf{V}_{\mathbb{C}}^{3}$ admits
a unique inverse given by $($\ref{7}$)$, satisfying $\mathbf{a}^{-1}\mathbf{a%
}=\mathbf{aa}^{-1}=\mathbf{1}_{\mathbf{V}}$.
\end{theorem}

\begin{proof}
These properties follow directly from the corresponding properties of the
mutually orthogonal component fields $\left\langle \mathbf{V}_{\mathbb{C}%
}\right\rangle _{\alpha }$.
\end{proof}

\begin{corollary}
The algebraic structure $(\mathbf{V}_{\mathbb{C}}^{3},+,\star )$ is a
commutative field.
\end{corollary}

\begin{proof}
This is an immediate consequence of \textit{Theorem} 1.
\end{proof}

\begin{remark}
Since every complex vector is represented by an ordered pair consisting of a
real vector and an imaginary vector, the existence of the distinguished unit
pair $(\mathbf{1}_{\mathbf{V}},i\mathbf{0})$ serves as the multiplicative
identity of the geometric product, exactly as the complex unit $(1,i0)$ does
in the field of complex numbers $%
\mathbb{C}
$. Consequently, the underlying real of $\mathbf{V}_{\mathbb{C}}^{3}$ is
isomorphic to $\mathbf{V}^{3}\oplus \mathbf{V}^{3}$. Thus, the $3\mathbb{D}$
field of complex vectors $\mathbf{V}_{\mathbb{C}}^{3}$ may be viewed as the
direct sum of two real $3\mathbb{D}$ Euclidean vector spaces endowed with
the commutative geometric product.
\end{remark}

\begin{theorem}[Symmetric--antisymmetric decomposition]
Every geometric product admits the unique decomposition%
\begin{equation}
\mathbf{a\bar{b}}=\mathbf{\bar{a}\circ \bar{b}}+\mathbf{\bar{a}\wedge \bar{b}%
}\text{,}  \label{9}
\end{equation}%
where $\mathbf{a\circ b}=\mathbf{\bar{a}\circ \bar{b}}=(\mathbf{\bar{a}b}+%
\mathbf{a\bar{b})}/2$ and~$\mathbf{a\wedge b}=\mathbf{\bar{b}\wedge \bar{a}}%
=(\mathbf{\bar{a}b}-\mathbf{a\bar{b}})/2$ are the symmetric and
antisymmetric parts, respectively.
\end{theorem}

\begin{theorem}[Geometric interpretation]
For arbitrary vectors $\mathbf{a},\mathbf{b\in V}_{\mathbb{C}}^{3}$, such
that $\mathbf{a}=\langle a\mathbf{e}+\hat{a}\mathbf{\hat{e}\rangle }_{\alpha
}\mathbf{e}^{\alpha }$ and $\mathbf{b}=\langle b\mathbf{e}+\hat{b}\mathbf{%
\hat{e}\rangle }_{\alpha }\mathbf{e}^{\alpha }$, the symmetric and
antisymmetric parts of the geometric product $\mathbf{\bar{a}b}$ are given
by 
\begin{equation}
\mathbf{a\circ b}=\langle \mathbf{a\circ b\rangle }_{\alpha }\mathbf{e}%
^{\alpha }=\langle (\mathbf{\bar{a}}\cdot \mathbf{b)e\rangle }_{\alpha }%
\mathbf{e}^{\alpha }=\langle (ab+\hat{a}\hat{b})\mathbf{e}\rangle _{\alpha }%
\mathbf{e}^{\alpha }\text{ and}  \label{10}
\end{equation}%
\begin{equation}
\mathbf{a}\wedge \mathbf{b}=\langle \mathbf{a\wedge b\rangle }_{\alpha }%
\mathbf{e}^{\alpha }=\langle (\mathbf{a}\times \mathbf{b)}\times \mathbf{%
e\rangle }_{\alpha }\mathbf{e}^{\alpha }=[a,\hat{b}]_{\alpha }\mathbf{\hat{e}%
}^{\alpha }\text{,}  \label{11}
\end{equation}%
where $a\hat{b}-\hat{a}b$ is denoted by the bracket $[a,\hat{b}]$.
\end{theorem}

\begin{remark}
The symmetric part represents the vector analogue of the inner product,
whereas the antisymmetric part represents the vector analogue of the outer
product.
\end{remark}

\subsection{Differential forms in the field $\mathbf{V}_{\mathbb{C}}$}

To express a complex vector $\mathbf{\varrho \in V}_{\mathbb{C}}$ in polar
form, we first introduce the vector analogue of \textit{Euler's} formula%
\begin{equation}
\text{c}\mathbf{\hat{e}}\text{s}\mathtt{~}\mathbf{\cdot }=\mathbf{e}\cos 
\mathbf{\cdot +~\hat{e}}\sin \mathbf{\cdot }\text{.}  \label{12}
\end{equation}

Let $\mathbf{\varrho }_{0}$ $=$ c$\mathbf{\hat{e}}$s$\varphi =\exp (\mathbf{%
\hat{e}}\varphi )$ denote the radial unit complex vector, where $\exp (%
\mathbf{\hat{e}}\cdot )$ is the exponential representation of the operator c$%
\mathbf{\hat{e}}$s$\mathtt{~}\mathbf{\cdot }$. Every complex vector in $%
\mathbf{V}_{\mathbb{C}}$ admits the polar representation $\mathbf{\varrho }%
=\varrho \mathbf{\varrho }_{0}$. Consequently, $\log \mathbf{\varrho }=\ln
\varrho \mathbf{e}+\varphi \mathbf{\hat{e}}$, while Log$~\mathbf{\varrho }$%
\textbf{\ }$=\ln \varrho \mathbf{e}+(\varphi \pm 2\pi n)\mathbf{\hat{e}}$, $%
n\in 
\mathbb{N}
$, is the corresponding multivalued logarithm. If $\mathbf{\varrho }_{0\bot
}=\mathbf{\varrho }_{0}\mathbf{\hat{e}}$, then $\mathbf{\hat{\varrho}}_{0}=-%
\mathbf{\varrho }_{0\bot }^{-1}=$ s$\mathbf{\hat{e}}$c$\varphi =\mathbf{\hat{%
e}\bar{\varrho}}_{0}$. These four unit vectors form the direct $(\mathbf{%
\varrho }_{0},\mathbf{\varrho }_{0\bot })$ and inverse $(\mathbf{\bar{\varrho%
}}_{0},\mathbf{\hat{\varrho}}_{0})$ polar bases of $\mathbf{V}_{\mathbb{C}}$%
. For an arbitrary vector $\mathbf{a}\in \mathbf{V}_{\mathbb{C}}$, the
geometric product $\mathbf{a\varrho }_{0}$ rotates the vector $\mathbf{a}$
by the angle $\varphi $, in the positive mathematical direction, and the
geometric product $\mathbf{a\varrho }_{0\bot }$ by the angle $\pi /2+\varphi 
$. Similarly, the geometric products $\mathbf{a\bar{\varrho}}_{0}$ and $%
\mathbf{a\hat{\varrho}}_{0}$ rotate $\mathbf{a}$ by the angles $-\varphi $
and $\pi /2-\varphi $, respectively, in the positive mathematical direction.

In polar coordinates the differential operator takes the form $d=d\varrho
\partial _{\varrho }+d\varphi \partial _{\varphi }$. Consequently, $d\mathbf{%
\varrho }=d\varrho \mathbf{\varrho }_{0}+d\varphi \mathbf{\varrho }_{\bot }$
and $d\mathbf{\varrho }_{\bot }=\mathbf{\hat{e}}d\mathbf{\varrho }=d\varrho 
\mathbf{\varrho }_{0\bot }-d\varphi \mathbf{\varrho }$, Similarly, $d\mathbf{%
\hat{\varrho}}=\mathbf{\hat{e}}d\mathbf{\bar{\varrho}}=\mathbf{\hat{e}(}%
d\varrho \mathbf{\bar{\varrho}}_{0}-d\varphi \mathbf{\hat{\varrho})}%
=d\varrho \mathbf{\hat{\varrho}}_{0}+d\varphi \mathbf{\bar{\varrho}}$.
Furthermore, since $2\varrho \cos \varphi \mathbf{e}=\mathbf{\varrho }+%
\mathbf{\bar{\varrho}}$ and $2\varrho \sin \varphi \mathbf{\hat{e}}=\mathbf{%
\varrho -\bar{\varrho}}$, we introduce the following vector differential
operators, which are the vector analogues of the classical \textit{Wirtinger}
operators \cite{Tung}, 
\begin{equation}
\boldsymbol{\eth }_{\mathbf{\varrho }}=\partial _{\mathbf{\varrho }}\varrho
\partial _{\varrho }+\partial _{\mathbf{\varrho }}\varphi \partial _{\varphi
}=\frac{1}{2}(\mathbf{\bar{\varrho}}_{0}\partial _{\varrho }-\frac{\mathbf{%
\hat{\varrho}}_{0}}{\varrho }\partial _{\varphi })\text{ and }\boldsymbol{%
\eth }_{\mathbf{\bar{\varrho}}}=\boldsymbol{\bar{\eth }}_{\mathbf{\varrho }}=%
\frac{1}{2}(\mathbf{\varrho }_{0}\partial _{\varrho }+\frac{\mathbf{\varrho }%
_{0\mathbf{\bot }}}{\varrho }\partial _{\varphi })\text{,}  \label{13}
\end{equation}%
where $\partial _{\mathbf{\varrho }}\varphi =\cos ^{2}\varphi \partial _{%
\mathbf{\varrho }}\tan \varphi =[(\mathbf{\mathbf{\varrho }+\bar{\varrho})}%
/(2\varrho )]^{2}[(-2\mathbf{\bar{\varrho})/(\varrho }+\mathbf{\bar{\varrho}}%
)^{2}]\mathbf{\hat{e}}=-\mathbf{\hat{\varrho}}_{0}/(2\varrho )$ and $%
2\partial _{\mathbf{\varrho }}\varrho =\partial _{\mathbf{\varrho }}\varrho
^{2}/\varrho =\mathbf{\bar{\varrho}}_{0}$. Consequently, the differential
operator $2\boldsymbol{\eth }_{\mathbf{\bar{\varrho}}}$ is the gradient ($%
\func{grad}$) operator, and the symmetric and antisymmetric components of
the geometric product $2\boldsymbol{\eth }_{\mathbf{r}}\boldsymbol{F}$ are
the vector representation of the divergence ($\func{div}$) and $\func{curl}$
of the vector field $\boldsymbol{F}=F\mathbf{\varrho }_{0}+F_{\bot }\mathbf{%
\varrho }_{0\bot }\in \mathbf{V}_{\mathbb{C}}$, respectively, since%
\begin{equation}
2\boldsymbol{\eth }_{\mathbf{\varrho }}\boldsymbol{F}=\frac{\mathbf{\bar{%
\varrho}}_{0}\mathbf{\varrho }_{0}}{\varrho }[\partial _{\varrho }(\varrho
F)+\partial _{\varphi }F_{\bot }]+\frac{\mathbf{\bar{\varrho}}_{0}\mathbf{%
\varrho }_{0\bot }}{\varrho }[\partial _{\varrho }(\varrho F_{\bot
})-\partial _{\varphi }F]=  \label{14}
\end{equation}%
\begin{equation*}
=\mathbf{\bar{\varrho}}_{0}\mathbf{\varrho }_{0}\func{div}\boldsymbol{F}+%
\func{curl}\boldsymbol{F}\times \mathbf{\bar{\varrho}}_{0}\mathbf{\varrho }%
_{0}\text{,}
\end{equation*}%
where $\func{div}\boldsymbol{F}=[\partial _{\varrho }(\varrho F)+\partial
_{\varphi }F_{\bot }]/\varrho $ and $\func{curl}\boldsymbol{F}=[\partial
_{\varrho }(\varrho F_{\bot })-\partial _{\varphi }F]\mathbf{n}/\varrho $.

\begin{remark}
The differentiation rules for both geometric products and geometric
quotients coincide with the classical product and quotient rules of
differentiation. Thus, $d(\mathbf{\varrho }/\mathbf{\bar{\varrho})}=(\mathbf{%
\bar{\varrho}}d\mathbf{\varrho -\varrho }d\mathbf{\bar{\varrho}})/\mathbf{%
\bar{\varrho}}^{2}$. Indeed,%
\begin{equation}
d\frac{\mathbf{\varrho }}{\mathbf{\bar{\varrho}}}=(d\frac{1}{\varrho ^{2}})%
\mathbf{\varrho }^{2}+\frac{1}{\varrho ^{2}}d\mathbf{\varrho }^{2}=-2(\frac{%
\mathbf{\varrho }^{2}}{\varrho ^{3}}d\varrho -\frac{1}{\varrho ^{2}}\mathbf{%
\varrho }d\mathbf{\varrho )=}  \label{15}
\end{equation}%
\begin{equation*}
=-2[\frac{\mathbf{\varrho }^{2}}{2\varrho ^{4}}(\mathbf{\bar{\varrho}}d%
\mathbf{\varrho +\varrho }d\mathbf{\bar{\varrho}})-\frac{1}{\varrho ^{2}}%
\mathbf{\varrho }d\mathbf{\varrho ]=}-2[\frac{1}{2\mathbf{\bar{\varrho}}^{2}}%
(\mathbf{\bar{\varrho}}d\mathbf{\varrho +\varrho }d\mathbf{\bar{\varrho}})-%
\frac{1}{\mathbf{\bar{\varrho}}^{2}}\mathbf{\bar{\varrho}}d\mathbf{\varrho ]}%
=\frac{\mathbf{\bar{\varrho}}d\mathbf{\varrho -\varrho }d\mathbf{\bar{\varrho%
}}}{\mathbf{\bar{\varrho}}^{2}}\text{.}
\end{equation*}%
\qquad \qquad
\end{remark}

\begin{definition}
The geometric product $d\mathbf{\varrho }\boldsymbol{\eth }_{\mathbf{\varrho 
}}$ defines the operator of a $1\mathbb{D}$ differential form.
\end{definition}

The symmetric and antisymmetric components of $d\mathbf{\varrho }\boldsymbol{%
\eth }_{\mathbf{\varrho }}$ are given by:%
\begin{equation}
d\mathbf{\varrho }\boldsymbol{\eth }_{\mathbf{\varrho }}+d\mathbf{\bar{%
\varrho}}\boldsymbol{\eth }_{\mathbf{\bar{\varrho}}}=2d\mathbf{\varrho }%
\circ \boldsymbol{\eth }_{\mathbf{\bar{\varrho}}}=\mathbf{e}\left( d\varrho
\partial _{\varrho }+d\varphi \partial _{\varphi }\right) \text{ and}
\label{16}
\end{equation}%
\begin{equation*}
d\mathbf{\varrho }\boldsymbol{\eth }_{\mathbf{\varrho }}-d\mathbf{\bar{%
\varrho}}\boldsymbol{\eth }_{\mathbf{\bar{\varrho}}}=-2d\mathbf{\varrho }%
\wedge \boldsymbol{\eth }_{\mathbf{\bar{\varrho}}}=\mathbf{\hat{e}}(\varrho
d\varphi \partial _{\varrho }-\frac{1}{\varrho }d\varrho \partial _{\varphi
})\text{.}
\end{equation*}%
Therefore,%
\begin{equation}
2\mathbf{d\varrho }\boldsymbol{\eth }_{\mathbf{\varrho }}=2\mathbf{\varrho }%
_{0}d\mathbf{\varrho }\boldsymbol{\eth }_{\mathbf{\varrho }}=\mathbf{\varrho 
}_{0}\left( d\varrho \partial _{\varrho }+d\varphi \partial _{\varphi
}\right) +\mathbf{\varrho }_{0\bot }(\varrho d\varphi \partial _{\varrho }-%
\frac{1}{\varrho }d\varrho \partial _{\varphi })=\mathbf{d}+\mathbf{d}_{\bot
}\text{,}  \label{17}
\end{equation}%
where $\mathbf{d=\varrho }_{0}d$ and $\mathbf{d}_{\bot }=\mathbf{\varrho }%
_{0\bot }\hat{d}$ denote the radial and transverse components of the
operator $2\mathbf{\varrho }_{0}d\mathbf{\varrho }\boldsymbol{\eth }_{%
\mathbf{\varrho }}$, respectively. Using the identity $2\varphi \mathbf{\hat{%
e}}=$ $\log (\mathbf{\varrho }/\mathbf{\bar{\varrho})}$, we obtain 
\begin{equation}
\mathbf{d}\varphi =\mathbf{\varrho }_{0}d\varphi =\mathbf{\varrho }_{0}(%
\boldsymbol{\eth }_{\mathbf{\varrho }}\varphi d\mathbf{\varrho }+\boldsymbol{%
\eth }_{\mathbf{\bar{\varrho}}}\varphi d\mathbf{\bar{\varrho}})=-\frac{%
\mathbf{\varrho }_{0\bot }}{2}d\log (\frac{\mathbf{\varrho }}{\mathbf{\bar{%
\varrho}}})\text{ and}  \label{18}
\end{equation}%
\begin{equation*}
\mathbf{d}\varrho =\mathbf{\varrho }_{0}d\varrho =\mathbf{\varrho }%
_{0}\left( \boldsymbol{\eth }_{\mathbf{\varrho }}\varrho d\mathbf{\varrho }+%
\boldsymbol{\eth }_{\mathbf{\bar{\varrho}}}\varrho d\mathbf{\bar{\varrho}}%
\right) =\mathbf{\varrho }_{0}d(\mathbf{\varrho \bar{\varrho})}^{\frac{1}{2}%
}=\frac{\mathbf{\varrho }}{2}d\log (\mathbf{\varrho \bar{\varrho})}\text{.}
\end{equation*}%
Combining the equations (\ref{18}) yields the vector differential identity,%
\begin{equation}
\mathbf{\varrho }_{0\bot }d\varrho d\varphi =\frac{\mathbf{\varrho }}{4}%
d\log (\mathbf{\varrho \bar{\varrho})}d\log (\frac{\mathbf{\varrho }}{%
\mathbf{\bar{\varrho}}})=\frac{\mathbf{\varrho }}{4\varrho ^{4}}(\mathbf{%
\bar{\varrho}}d\mathbf{\varrho }+\mathbf{\varrho }d\mathbf{\bar{\varrho}})(%
\mathbf{\bar{\varrho}}d\mathbf{\varrho -\varrho }d\mathbf{\bar{\varrho}})=
\label{19}
\end{equation}%
\begin{equation*}
=\frac{\mathbf{\varrho }}{4\varrho ^{2}}(\mathbf{\bar{d}\varrho }+\mathbf{d%
\bar{\varrho}})(\mathbf{\bar{d}\varrho -d\bar{\varrho}})=\frac{\mathbf{%
\varrho }}{4\varrho ^{2}}[(\mathbf{\bar{d}\varrho })^{2}-(\mathbf{d\bar{%
\varrho}})^{2}]\text{.}
\end{equation*}%
Consequently,%
\begin{equation}
\frac{\mathbf{\varrho }_{0}}{2}(\mathbf{d\bar{\varrho}}\wedge \mathbf{\bar{d}%
\varrho )}=\frac{\mathbf{\varrho }}{4\varrho }[(\mathbf{\bar{d}\varrho }%
)^{2}-(\mathbf{d\bar{\varrho}})^{2}]=\mathbf{\varrho }_{\bot }d\varrho
d\varphi \text{,}  \label{20}
\end{equation}%
since $\overline{\mathbf{d\bar{\varrho}}}=\mathbf{\bar{\varrho}}_{0}d\mathbf{%
\varrho }=\mathbf{\bar{d}}\mathbf{\varrho }$. The same identity can also be
derived directly from the \textit{Jacobian} determinant $\mathbf{J}$ of the
following bijective mapping $2\ln \varrho \mathbf{e}=\log (\mathbf{\varrho 
\bar{\varrho}})$ and $2\varphi \mathbf{\hat{e}}=$ $\log (\mathbf{\varrho }/%
\mathbf{\bar{\varrho}})$,%
\begin{equation}
\mathbf{J}=\left\vert 
\begin{array}{cc}
\mathbf{\partial }_{\varrho }\log (\mathbf{\varrho \bar{\varrho})} & \mathbf{%
\partial }_{\varphi }\log (\mathbf{\varrho \bar{\varrho})} \\ 
\mathbf{\partial }_{\varrho }\log (\mathbf{\varrho }/\mathbf{\bar{\varrho})}
& \mathbf{\partial }_{\varphi }\log (\mathbf{\varrho }/\mathbf{\bar{\varrho})%
}%
\end{array}%
\right\vert =\left\vert 
\begin{array}{cc}
2\varrho ^{-1}\mathbf{e} & \mathbf{0} \\ 
\mathbf{0} & 2\mathbf{\hat{e}}%
\end{array}%
\right\vert =\frac{4}{\varrho }\mathbf{\hat{e}}\text{.}  \label{21}
\end{equation}%
Consequently, $4\mathbf{\varrho }_{0\bot }d\varrho d\varphi =\mathbf{\varrho
J}d\varrho d\varphi =\mathbf{\varrho }d\log (\mathbf{\varrho \bar{\varrho})}%
d\log (\mathbf{\varrho }/\mathbf{\bar{\varrho}})$, which leads to (\ref{18}%
). The complex vector $d\mathbf{S}=(\mathbf{d\bar{\varrho}}\wedge \mathbf{%
\bar{d}\varrho )}/2=\varrho d\varrho d\varphi \mathbf{\varrho }_{0\bot }%
\mathbf{\bar{\varrho}}_{0}=dS\mathbf{\hat{e}}$ represents the \textit{%
Lebesgue} vector measure of the infinitesimal surface element in $\mathbf{V}%
_{\mathbb{C}}$.

\begin{definition}
Let $\boldsymbol{\eth }_{\mathbf{\varrho \bar{\varrho}}}^{2}=\boldsymbol{%
\eth }_{\mathbf{\bar{\varrho}}}\boldsymbol{\eth }_{\mathbf{\varrho }}$. The
geometric product $d\mathbf{S}\boldsymbol{\eth }_{\mathbf{\varrho \bar{%
\varrho}}}^{2}$ defines the operator of a $2\mathbb{D}$ differential form.
\end{definition}

\subsection{Fundamental theorem of integral calculus in $\mathbf{V}_{\mathbb{%
C}}$}

Let $\gamma $ be a smooth closed \textit{Jordan} curve bounding a simply
connected region $G\subset \mathbf{V}_{\mathbb{C}}$. Let $\mathbf{\varrho }%
_{\gamma }\in \gamma $ be surrounded by a circle $c(\mathbf{\varrho }%
_{\gamma },\varepsilon )$ of arbitrarily small radius $\varepsilon $,
intersecting $\gamma $ at the points $\mathbf{\varrho }_{\gamma _{1}}$ and $%
\mathbf{\varrho }_{\gamma _{2}}$. Likewise, let a point $\mathbf{\varrho }%
_{G}$ inside $G$ be surrounded by a circle $c(\mathbf{\varrho }%
_{G},\varepsilon )$. The circles $c(\mathbf{\varrho }_{\gamma },\varepsilon
) $ and $c(\mathbf{\varrho }_{\gamma },\varepsilon )$ are connected by two
parallel line segments $l_{_{\delta _{\varepsilon }}}^{1}$ and $l_{_{\delta
_{\varepsilon }}}^{2}$, separated by a distance $\delta _{\varepsilon }\ll
\varepsilon $. Denote by $S_{\varepsilon }$ the region bounded by the two
line segments together with the corresponding arcs of the circle's
perimeters $\partial c(\mathbf{\varrho }_{\gamma },\varepsilon )$ and $%
\partial c(\mathbf{\varrho }_{G},\varepsilon )$, whose endpoints are the
intersections of these perimeters with the line segments. The region $%
S_{\varepsilon }$ contains both points $\mathbf{\varrho }_{\gamma _{1}}$ and 
$\mathbf{\varrho }_{\gamma _{2}}$ and will be referred to as the residue
region. Let $G_{\varepsilon }^{+}=G\cup S_{\varepsilon }$. Then the contour
integration operators are defined by 
\begin{equation}
\int_{\gamma _{\varepsilon }^{+}}^{\circlearrowleft }d\mathbf{\varrho }%
\boldsymbol{\eth }_{\mathbf{\varrho }}=\int_{\partial G_{\varepsilon
}^{+}}^{\circlearrowleft }d\mathbf{\varrho }\boldsymbol{\eth }_{\mathbf{%
\varrho }}\text{ and }\int_{\gamma _{\varepsilon }^{-}}^{\circlearrowleft }d%
\mathbf{\varrho }\boldsymbol{\eth }_{\mathbf{\varrho }}=\int_{\partial
G_{\varepsilon }^{+}}^{\circlearrowleft }d\mathbf{\varrho }\boldsymbol{\eth }%
_{\mathbf{\varrho }}-\int_{\partial c(\mathbf{\varrho }_{\gamma
},\varepsilon )}^{\circlearrowleft }d\mathbf{\varrho }\boldsymbol{\eth }_{%
\mathbf{\varrho }}\text{,}  \label{22}
\end{equation}%
where $\partial G_{\varepsilon }^{+}=\gamma _{\varepsilon }^{+}$ denotes the
boundary of $G_{\varepsilon }^{+}$.

By the additivity of definite integrals, taking the limit as $\varepsilon
\rightarrow 0^{+}$,

\begin{equation}
vt\int_{\gamma ^{+}}^{\circlearrowleft }d\mathbf{\varrho \boldsymbol{\eth }%
_{\varrho }}=\lim_{\varepsilon \rightarrow 0^{+}}\int_{\gamma _{\varepsilon
}^{+}}^{\circlearrowleft }d\mathbf{\varrho \boldsymbol{\eth }_{\varrho }=}%
\lim_{\varepsilon \rightarrow 0^{+}}(\int_{\partial G_{\delta _{\varepsilon
}}}^{\circlearrowleft }d\mathbf{\varrho \boldsymbol{\eth }_{\varrho }}%
+\int_{\partial S_{\varepsilon }}^{\circlearrowleft }d\mathbf{\varrho 
\boldsymbol{\eth }_{\varrho }})=  \label{23}
\end{equation}%
\begin{equation*}
=vp\int_{\gamma }^{\circlearrowleft }d\mathbf{\varrho \boldsymbol{\eth }%
_{\varrho }+}\lim_{\varepsilon \rightarrow 0^{+}}\int_{\partial c(\mathbf{%
\varrho }_{\gamma },\varepsilon )}^{\overset{\curvearrowleft }{\mathbf{%
\varrho }_{\gamma _{2}}\mathbf{\varrho }_{\gamma _{1}}}}d\mathbf{\varrho 
\boldsymbol{\eth }_{\varrho }}=vt\int_{\gamma ^{-}}^{\circlearrowleft }d%
\mathbf{\varrho \boldsymbol{\eth }_{\varrho }+}\lim_{\varepsilon \rightarrow
0^{+}}\int_{\partial c(\mathbf{\varrho }_{\gamma },\varepsilon
)}^{\circlearrowleft }d\mathbf{\varrho \boldsymbol{\eth }_{\varrho }}\text{,}
\end{equation*}%
where $vp$ and $vt$ denote the principal and total integral values,
respectively \cite{Sar1}-\cite{Sar6}, and $G_{\delta _{\varepsilon }}=$ $%
G\backslash int.S_{\varepsilon }$, where $int.S_{\varepsilon }$ denotes the
interior of $S_{\varepsilon }$.

On the one hand, as a vector analogue of the areolar (weak) derivative
introduced by \textit{Pompeiu} in complex analysis\textit{\ }\cite{Pomp},
the vector differential operator $\boldsymbol{\eth }_{\mathbf{\varrho \bar{%
\varrho}}}^{2}$ is defined by%
\begin{equation}
\lim_{G\rightarrow \mathbf{\varrho }_{G}}\frac{1}{2\mathbf{S}_{G}}%
\int_{\partial G}^{\circlearrowleft }d\mathbf{\varrho }\boldsymbol{\eth }_{%
\mathbf{\varrho }}=\boldsymbol{\eth }_{\mathbf{\varrho \bar{\varrho}}%
}^{2}\left\vert _{\mathbf{\varrho }_{G}}\right. \text{,}  \label{24}
\end{equation}%
where $2\mathbf{S}_{G}=\int_{\partial G}^{\circlearrowleft }\mathbf{\varrho }%
\wedge d\mathbf{\varrho }$. Consequently, $\lim_{G\rightarrow \mathbf{%
\varrho }_{G}}\int_{\partial G}^{\circlearrowleft }d\mathbf{\varrho }%
\boldsymbol{\eth }_{\mathbf{\varrho }}=2d\mathbf{S}\boldsymbol{\eth }_{%
\mathbf{\varrho \bar{\varrho}}}^{2}\left\vert _{\mathbf{\varrho }%
_{G}}\right. $. On the other hand, by the \textit{Kelvin-Stokes }(\textit{%
Green}) theorem \cite{Mars},%
\begin{equation}
\lim_{\varepsilon \rightarrow 0^{+}}\int_{\partial G_{\delta _{\varepsilon
}}}^{\circlearrowleft }d\mathbf{\varrho \boldsymbol{\eth }_{\varrho }}=\frac{%
1}{2}\lim_{\varepsilon \rightarrow 0^{+}}\int_{\partial G_{\delta
_{\varepsilon }}}^{\circlearrowleft }[(d\mathbf{\varrho }\circ \func{grad})%
\text{ }\mathbf{-}\text{ }(d\mathbf{\varrho }\wedge \func{grad})]=
\label{25}
\end{equation}%
\begin{equation*}
=\frac{1}{2}\lim_{\varepsilon \rightarrow 0^{+}}\int_{G_{\delta
_{\varepsilon }}}\mathbf{\mathbf{\hat{\varrho}}}drd\varphi (\mathbf{\mathbf{%
\varrho }}_{0}\func{div}\func{grad}\mathbf{+}\text{ }\mathbf{\mathbf{\varrho 
}}_{0}\times \func{curl}\func{grad})=vp\int_{G}(\mathbf{d\bar{\varrho}}%
\wedge \mathbf{\bar{d}\varrho )}\boldsymbol{\eth }_{\mathbf{\varrho \bar{%
\varrho}}}^{2}\text{,}
\end{equation*}%
since%
\begin{equation}
\boldsymbol{\eth }_{\mathbf{\varrho \bar{\varrho}}}^{2}=\boldsymbol{\eth }_{%
\mathbf{\bar{\varrho}}}\boldsymbol{\eth }_{\mathbf{\varrho }}=\frac{\mathbf{%
\bar{\varrho}}_{0}}{4\varrho ^{2}}[\mathbf{\mathbf{\varrho }}_{0}[\varrho
\partial _{\varrho }(\varrho \partial _{\varrho })+\partial _{\varphi
^{2}}^{2}]+\mathbf{\mathbf{\varrho }}_{\perp }(\partial _{\varrho \varphi
}^{2}-\partial _{\varphi \varrho }^{2})]=  \label{26}
\end{equation}%
\begin{equation*}
=\frac{\mathbf{\bar{\varrho}}_{0}\mathbf{\mathbf{\varrho }}_{0}}{\varrho ^{2}%
}(\mathbf{\varrho }\boldsymbol{\eth }_{\mathbf{\varrho }}\cdot \mathbf{\bar{%
\varrho}}\boldsymbol{\eth }_{\mathbf{\bar{\varrho}}})+\frac{\mathbf{\bar{%
\varrho}}_{0}\mathbf{\mathbf{\varrho }}_{0}}{\varrho ^{2}}\times (\mathbf{%
\bar{\varrho}}\boldsymbol{\eth }_{\mathbf{\bar{\varrho}}}\times \mathbf{\bar{%
\varrho}}\boldsymbol{\eth }_{\mathbf{\bar{\varrho}}}\mathbf{)}=\frac{\mathbf{%
\bar{\varrho}}_{0}}{4}(\mathbf{\mathbf{\varrho }}_{0}\func{div}\func{grad}%
\mathbf{+}\text{ }\mathbf{\mathbf{\varrho }}_{0}\times \func{curl}\func{grad}%
)\text{,}
\end{equation*}%
where%
\begin{equation}
\func{div}\func{grad}=\frac{4}{\varrho ^{2}}(\mathbf{\varrho }\boldsymbol{%
\eth }_{\mathbf{\varrho }}\cdot \mathbf{\bar{\varrho}}\boldsymbol{\eth }_{%
\mathbf{\bar{\varrho}}})=\frac{1}{\varrho ^{2}}[\varrho \partial _{\varrho
}(\varrho \partial _{\varrho })+\partial _{\varphi ^{2}}^{2}]\text{ and}
\label{27}
\end{equation}%
\begin{equation*}
\func{curl}\func{grad}=\frac{4}{\varrho ^{2}}(\mathbf{\bar{\varrho}}%
\boldsymbol{\eth }_{\mathbf{\bar{\varrho}}}\times \mathbf{\bar{\varrho}}%
\boldsymbol{\eth }_{\mathbf{\bar{\varrho}}})=\frac{\mathbf{\varrho }%
_{0}\times \mathbf{\mathbf{\varrho }}_{\bot }}{\varrho ^{2}}(\partial
_{\varphi \varrho }^{2}-\partial _{\varrho \varphi }^{2})=\mathbf{0}\text{.}
\end{equation*}%
Consequently,%
\begin{equation}
vt\int_{\gamma ^{+}}^{\circlearrowleft }d\mathbf{\varrho }\boldsymbol{\eth }%
_{\mathbf{\varrho }}=vt\int_{G^{+}}(\mathbf{d\bar{\varrho}}\wedge \mathbf{%
\bar{d}\varrho )}\boldsymbol{\eth }_{\mathbf{\varrho \bar{\varrho}}}^{2}%
\text{,}  \label{28}
\end{equation}%
where $vt\int_{G^{+}}(\mathbf{d\bar{\varrho}}\wedge \mathbf{\bar{d}\varrho )}%
\boldsymbol{\eth }_{\mathbf{\varrho \bar{\varrho}}}^{2}$ is the total
surface integration operator defined by 
\begin{equation}
vt\int_{G^{+}}(\mathbf{d\bar{\varrho}}\wedge \mathbf{\bar{d}\varrho )}%
\boldsymbol{\eth }_{\mathbf{\varrho \bar{\varrho}}}^{2}=vp\int_{G}(\mathbf{d%
\bar{\varrho}}\wedge \mathbf{\bar{d}\varrho )}\boldsymbol{\eth }_{\mathbf{%
\varrho \bar{\varrho}}}^{2}+\lim_{\varepsilon \rightarrow
0^{+}}\int_{\partial S_{\varepsilon }}^{\circlearrowleft }d\mathbf{\varrho 
\boldsymbol{\eth }_{\varrho }}\text{.}  \label{29}
\end{equation}%
As $\varepsilon \rightarrow 0^{+}$, the vector integral operator $%
\int_{\partial S_{\varepsilon }}^{\circlearrowleft }d\mathbf{\varrho 
\boldsymbol{\eth }_{\mathbf{\varrho }}}$ satisfies

\begin{equation}
\lim_{\varepsilon \rightarrow 0^{+}}\int_{\partial S_{\varepsilon
}}^{\circlearrowleft }d\mathbf{\varrho \boldsymbol{\eth }_{\mathbf{\varrho }}%
}=2\pi \mathbf{\hat{e}}Res\mathbf{\boldsymbol{\eth }_{\varrho }}\left\vert _{%
\mathbf{\varrho }_{G}}\right. +\int_{l^{+}}d\mathbf{\varrho \boldsymbol{\eth 
}_{\mathbf{\varrho }}-}\int_{l^{-}}d\mathbf{\varrho \boldsymbol{\eth }_{%
\mathbf{\varrho }}}+2\pi \mathbf{\hat{e}}Res\mathbf{\boldsymbol{\eth }%
_{\varrho }}\left\vert _{\mathbf{\varrho }_{\gamma }}\right. \text{,}
\label{30}
\end{equation}%
where%
\begin{equation}
\int_{l^{+}}d\mathbf{\varrho \boldsymbol{\eth }_{\mathbf{\varrho }}}%
=\lim_{\varepsilon \rightarrow 0^{+}}\int_{l_{\delta _{\varepsilon }}^{1}}d%
\mathbf{\varrho \boldsymbol{\eth }_{\mathbf{\varrho }}}\text{, }\int_{l^{-}}d%
\mathbf{\varrho \boldsymbol{\eth }_{\mathbf{\varrho }}}=\lim_{\varepsilon
\rightarrow 0^{+}}\int_{l_{\delta _{\varepsilon }}^{2}}d\mathbf{\varrho 
\boldsymbol{\eth }_{\mathbf{\varrho }}}\text{, }  \label{31}
\end{equation}%
\begin{equation*}
2\pi \mathbf{\hat{e}}Res\mathbf{\boldsymbol{\eth }_{\varrho }}\left\vert _{%
\mathbf{\varrho }_{\gamma }}\right. =\lim_{\varepsilon \rightarrow
0^{+}}\int_{\partial c(\mathbf{\varrho }_{\gamma },\varepsilon
)}^{\circlearrowleft }d\mathbf{\varrho \boldsymbol{\eth }_{\varrho }}=2d%
\mathbf{S}\boldsymbol{\eth }_{\mathbf{\varrho \bar{\varrho}}}^{2}\left\vert
_{\mathbf{\varrho }_{\gamma }}\right. \text{ and}
\end{equation*}%
\begin{equation*}
2\pi \mathbf{\hat{e}}Res\mathbf{\boldsymbol{\eth }_{\varrho }}\left\vert _{%
\mathbf{\varrho }_{G}}\right. =\lim_{\varepsilon \rightarrow
0^{+}}\int_{\partial c(\mathbf{\varrho }_{G},\varepsilon
)}^{\circlearrowleft }d\mathbf{\varrho \boldsymbol{\eth }_{\varrho }}=2d%
\mathbf{S}\boldsymbol{\eth }_{\mathbf{\varrho \bar{\varrho}}}^{2}\left\vert
_{\mathbf{\varrho }_{G}}\right. \text{.}
\end{equation*}%
The limit integral operator $\lim_{\varepsilon \rightarrow
0^{+}}\int_{\partial S_{\varepsilon }}^{\circlearrowleft }d\mathbf{\varrho 
\boldsymbol{\eth }_{\mathbf{\varrho }}}$ is called the vector residue
operator in $G$. Thus, $\pi \mathbf{\hat{e}}Res\mathbf{\boldsymbol{\eth }%
_{\varrho }}=d\mathbf{S}\boldsymbol{\eth }_{\mathbf{\varrho \bar{\varrho}}%
}^{2}$, so the two notations may be used interchangeably.

\begin{remark}
The previous integral definition of the residue, given as the limit of a
contour integral, extends Poor's definition of the residue for non-analytic
functions in complex analysis to the vector setting, \cite{MiKe}.
\end{remark}

\begin{definition}
Let $\gamma $ be a smooth closed \textit{Jordan} curve bounding a simply
connected region $G\subset \mathbf{V}_{\mathbb{C}}$. A scalar or vector
field in $\mathbf{V}_{%
\mathbb{C}
}$, whose first-order partial derivatives are continuous throughout $G$, is
said to be regular in $G$.
\end{definition}

Since both points $\mathbf{\varrho }_{\gamma }\in \gamma $ and $\mathbf{%
\varrho }_{G}\in G$ were chosen arbitrarily, the preceding construction is
completely general. We may therefore formulate the fundamental theorem of
integral calculus in $\mathbf{V}_{%
\mathbb{C}
}$. Before doing so, let $\mathbf{\omega }_{\mathbf{\varrho }}$ denote the
vector differential form obtained by applying the operator $d\mathbf{\varrho 
\boldsymbol{\eth }_{\varrho }}$ to an arbitrary uniform scalar or vector
field $\mathbf{\bullet }$ in $\mathbf{V}_{%
\mathbb{C}
}$. In complex analysis \cite{MiKe}, the term uniform complex function
refers to a single-valued function. By analogy, the same terminology is
adopted here for scalar and vector fields in $\mathbf{V}_{%
\mathbb{C}
}$.

\begin{theorem}
Let $\gamma $ be a smooth closed \textit{Jordan} curve bounding a simply
connected region $G$ $\subset \mathbf{V}_{\mathbb{C}}$. Let $\mathbf{\bullet 
}$ be a uniform scalar or vector field in $\mathbf{V}_{\mathbb{C}}$, regular
on $G|E$, where $E\subset G$ is a finite set of singular points, and suppose
that the corresponding vector differential form $\mathbf{\omega }_{\mathbf{%
\varrho }}$ is totally integrable on $\gamma $. Then, 
\begin{equation}
vt\int_{\gamma ^{+}}^{\circlearrowleft }\mathbf{\omega }_{\mathbf{\varrho }%
}=2\pi \mathbf{\hat{e}}\sum_{\mathbf{\varrho }_{i}\in G}Res\mathbf{%
\boldsymbol{\eth }_{\varrho }\bullet }(\mathbf{\varrho }_{i})=vt\int_{G^{+}}%
\mathbf{\omega }_{\mathbf{S}}\text{,}  \label{32}
\end{equation}%
where $\mathbf{\omega }_{\mathbf{S}}=2d\mathbf{S}\boldsymbol{\eth }_{\mathbf{%
\varrho \bar{\varrho}}}^{2}\bullet $.
\end{theorem}

\begin{remark}
The integral formula $($\ref{32}$)$ may be regarded as a vector analogue of
the \textit{Cauchy-Pompeiu} integral formula \cite{Tuts}, with the
generalization arising from the use of the total integral value. Since $Res%
\mathbf{\boldsymbol{\eth }_{\varrho }\bullet }$ vanishes identically on $%
G\backslash E$, it follows that 
\begin{equation}
2\pi \mathbf{\hat{e}}\sum_{\mathbf{\varrho }_{i}\in G\backslash E}Res\mathbf{%
\boldsymbol{\eth }_{\varrho }\bullet }(\mathbf{\varrho }_{i})=vp\int_{G}%
\mathbf{\omega }_{\mathbf{S}}\text{ and}  \label{33}
\end{equation}%
\begin{equation*}
vt\int_{G^{+}}\mathbf{\omega }_{\mathbf{S}}=vp\int_{G}\mathbf{\omega }_{%
\mathbf{S}}+2\pi \mathbf{\hat{e}}\sum_{\mathbf{\varrho }_{i}\in E}Res\mathbf{%
\boldsymbol{\eth }_{\varrho }\bullet }(\mathbf{\varrho }_{i})\text{.}
\end{equation*}%
When surface integrals encounter a singularity, both the \textit{Cauchy}
principal-value $(vp)$ integral and the residue contribution, can diverge.
Their combination may therefore formally have the indeterminate form $\infty
-\infty $. In cases where the total integral value $vt\int_{\gamma
^{+}}^{\circlearrowleft }\mathbf{\omega }_{\mathbf{\varrho }}$ is finite,
this indeterminate expression is resolved into that finite value.
\end{remark}

\subsection{Integral calculus formulas in the field $\mathbf{V}_{\mathbb{C}%
}^{3}$}

Let $G_{\alpha }$ be simply connected regions bounded by smooth closed 
\textit{Jordan} curves $\gamma _{\alpha }$, obtained as the projections of a
smooth surface $S$ onto the component fields of the 3$\mathbb{D}$ field of
complex vectors $\mathbf{V}_{\mathbb{C}}^{3}$. If $\langle d\varrho \sin
\varphi +\varrho d\varphi \cos \varphi \rangle _{1}=\langle d\varrho \cos
\varphi -\varrho d\varphi \sin \varphi \rangle _{2}$, $\langle d\varrho \sin
\varphi +\varrho d\varphi \cos \varphi \rangle _{2}=\langle d\varrho \cos
\varphi -\varrho d\varphi \sin \varphi \rangle _{3}$ and $\langle d\varrho
\sin \varphi +\varrho d\varphi \cos \varphi \rangle _{3}=\langle d\varrho
\cos \varphi -\varrho d\varphi \sin \varphi \rangle _{1}$, then $d\mathbf{%
\varrho }=\langle d\varrho \mathbf{\varrho }_{0}+d\varphi \mathbf{\varrho }%
_{\bot }\mathbf{\rangle }_{\alpha }(\mathbf{\bar{\varrho}}_{0}\mathbf{%
\varrho }_{0})^{\alpha }$. Consider the following differential forms in $%
\mathbf{V}_{\mathbb{C}}^{3}$,%
\begin{equation}
\mathbf{\omega }_{\mathbf{\bar{\varrho}}}=\mathbf{F}d\mathbf{\bar{\varrho}}%
=\left\langle \mathbf{\boldsymbol{F}}d\mathbf{\bar{\varrho}}\right\rangle
_{\alpha }(\mathbf{\bar{\varrho}}_{0}\mathbf{\varrho }_{0})^{\alpha }\text{
and}  \label{34}
\end{equation}%
\begin{equation*}
\mathbf{\omega }_{\mathbf{\hat{S}}}=-2\mathbf{\boldsymbol{\eth }}_{\mathbf{%
\varrho }}\mathbf{F}d\mathbf{\hat{S}}=-2\langle \mathbf{\boldsymbol{\eth }}_{%
\mathbf{\varrho }}\mathbf{F}d\mathbf{\hat{S}\rangle }_{\alpha }(\mathbf{\bar{%
\varrho}}_{0}\mathbf{\varrho }_{0})^{\alpha }\text{,}
\end{equation*}%
where $\mathbf{F}=\mathbf{F}_{\alpha }(\mathbf{\bar{\varrho}}_{0}\mathbf{%
\varrho }_{0})^{\alpha }=\langle F\mathbf{\varrho }_{0}+F_{\bot }\mathbf{%
\varrho }_{0\bot }\mathbf{\rangle }_{\alpha }(\mathbf{\bar{\varrho}}_{0}%
\mathbf{\varrho }_{0})^{\alpha }$ denotes an arbitrary uniform vector field
in $\mathbf{V}_{\mathbb{C}}^{3}$ and $2d\mathbf{\hat{S}}_{\alpha
}=\left\langle \mathbf{d\bar{\varrho}}\wedge \mathbf{\bar{d}\varrho }%
\right\rangle _{\alpha }$.

By \textit{Theorem }5, if the vector field $\mathbf{F}$ is regular almost
everywhere on $S$, and the vector differential forms $\mathbf{\bar{\omega}}_{%
\mathbf{\varrho }}=\mathbf{\bar{F}}d\mathbf{\varrho }$ and $\mathbf{\omega }%
_{\mathbf{\bar{\varrho}}}$ are totally integrable on $\partial S$, then%
\begin{equation}
\langle vt\int_{\gamma ^{+}}^{\circlearrowleft }\mathbf{F}\circ d\mathbf{%
\varrho }\rangle _{\alpha }(\mathbf{\bar{\varrho}}_{0}\mathbf{\varrho }%
_{0})^{\alpha }=-2\langle vt\int_{G^{+}}(\mathbf{\boldsymbol{\eth }_{\bar{%
\varrho}}\wedge F)}d\mathbf{\hat{S}}\rangle _{\alpha }(\mathbf{\bar{\varrho}}%
_{0}\mathbf{\varrho }_{0})^{\alpha }\text{ and}  \label{35}
\end{equation}%
\begin{equation*}
\langle vt\int_{\gamma ^{+}}^{\circlearrowleft }\mathbf{F}\wedge d\mathbf{%
\varrho }\rangle _{\alpha }(\mathbf{\bar{\varrho}}_{0}\mathbf{\varrho }%
_{0})^{\alpha }=2\langle vt\int_{G^{+}}(\mathbf{\boldsymbol{\eth }_{\bar{%
\varrho}}\circ F)}d\mathbf{\hat{S}}\rangle _{\alpha }(\mathbf{\bar{\varrho}}%
_{0}\mathbf{\varrho }_{0})^{\alpha }\text{,}
\end{equation*}%
so that%
\begin{equation}
vt\int_{\partial S^{+}}^{\circlearrowleft }\mathbf{\bar{\omega}}_{\mathbf{%
\varrho }}=vt\int_{\partial S^{+}}^{\circlearrowleft }\mathbf{\bar{F}}d%
\mathbf{\varrho }=\langle vt\int_{\gamma ^{+}}^{\circlearrowleft }\mathbf{%
\bar{F}}d\mathbf{\varrho }\rangle _{\alpha }(\mathbf{\bar{\varrho}}_{0}%
\mathbf{\varrho }_{0})^{\alpha }=  \label{36}
\end{equation}%
\begin{equation*}
=2\langle vt\int_{G^{+}}\overline{\mathbf{\boldsymbol{\eth }}_{\mathbf{%
\varrho }}\mathbf{F}}d\mathbf{\hat{S}}\rangle _{\alpha }(\mathbf{\bar{\varrho%
}}_{0}\mathbf{\varrho }_{0})^{\alpha }=2vt\int_{S^{+}}\overline{\mathbf{%
\boldsymbol{\eth }}_{\mathbf{\varrho }}\mathbf{F}}d\mathbf{\hat{S}}%
=vt\int_{S^{+}}\mathbf{\bar{\omega}}_{\mathbf{\hat{S}}}\text{ and}
\end{equation*}%
\begin{equation}
vt\int_{\partial S^{+}}^{\circlearrowleft }\mathbf{\omega }_{\mathbf{\bar{%
\varrho}}}=vt\int_{\partial S^{+}}^{\circlearrowleft }\mathbf{F}d\mathbf{%
\bar{\varrho}}=\langle vt\int_{\gamma ^{+}}^{\circlearrowleft }\mathbf{F}d%
\mathbf{\bar{\varrho}}\rangle _{\alpha }(\mathbf{\bar{\varrho}}_{0}\mathbf{%
\varrho }_{0})^{\alpha }=  \label{37}
\end{equation}%
\begin{equation*}
=-2\langle vt\int_{G^{+}}\mathbf{\boldsymbol{\eth }}_{\mathbf{\varrho }}%
\mathbf{F}d\mathbf{\hat{S}}\rangle _{\alpha }(\mathbf{\bar{\varrho}}_{0}%
\mathbf{\varrho }_{0})^{\alpha }=-2vt\int_{S^{+}}\mathbf{\boldsymbol{\eth }}%
_{\mathbf{\varrho }}\mathbf{F}d\mathbf{\hat{S}}=vt\int_{S^{+}}\mathbf{\omega 
}_{\mathbf{\hat{S}}}\text{.}
\end{equation*}%
Hence, $vt\int_{\partial S^{+}}^{\circlearrowleft }\mathbf{\bar{\omega}}_{%
\mathbf{\varrho }}=vt\int_{S^{+}}\mathbf{\bar{\omega}}_{\mathbf{\hat{S}}}$
and $vt\int_{\partial S^{+}}^{\circlearrowleft }\mathbf{\omega }_{\mathbf{%
\bar{\varrho}}}=vt\int_{S^{+}}\mathbf{\omega }_{\mathbf{\hat{S}}}$
establishing the integral form of the fundamental theorem for both
differential forms.

Furthermore, using the identities 
\begin{equation}
\langle (\overline{\mathbf{\boldsymbol{\eth }}_{\mathbf{\varrho }}\mathbf{F}}%
-\mathbf{\boldsymbol{\eth }}_{\mathbf{\varrho }}\mathbf{F})d\mathbf{\hat{S}%
\rangle }_{\alpha }\cdot (\mathbf{\bar{\varrho}}_{0}\mathbf{\varrho }%
_{0})^{\alpha }=-\langle (\func{curl}\mathbf{F}\times \mathbf{\bar{\varrho}}%
_{0}\mathbf{\varrho }_{0})d\mathbf{\hat{S}}\rangle _{\alpha }\cdot (\mathbf{%
\bar{\varrho}}_{0}\mathbf{\varrho }_{0})^{\alpha }=  \label{38}
\end{equation}%
\begin{equation*}
=\left\langle \mathbf{\bar{\varrho}}_{0}\mathbf{\varrho }_{0}\times \func{%
curl}\mathbf{F}\right\rangle ^{\alpha }\cdot d\mathbf{\hat{S}}_{\alpha
}=\left\langle \func{curl}\mathbf{F}\right\rangle ^{\alpha }\cdot \langle d%
\mathbf{\hat{S}}\times \mathbf{\bar{\varrho}}_{0}\mathbf{\varrho }%
_{0}\rangle _{\alpha }=\func{curl}\mathbf{F}\cdot d\mathbf{\bar{S}}\text{ and%
}
\end{equation*}%
\begin{equation}
(\mathbf{\bar{\varrho}}_{0}\mathbf{\varrho }_{0})^{\alpha }\times \langle (%
\mathbf{\overline{\mathbf{\boldsymbol{\eth }_{\varrho }\mathbf{F}}}}+\mathbf{%
\mathbf{\boldsymbol{\eth }_{\varrho }\mathbf{F}})}d\mathbf{\hat{S}}\rangle
_{\alpha }=\left\langle \func{div}\mathbf{F\bar{\varrho}}_{0}\mathbf{\varrho 
}_{0}\right\rangle ^{\alpha }\times d\mathbf{\hat{S}}_{\alpha }\text{,}
\label{39}
\end{equation}%
where $\left\langle \func{div}\mathbf{F}\right\rangle _{\alpha
}=\left\langle [\partial _{\varrho }(\varrho F)+\partial _{\varphi }F_{\bot
}]/\varrho \right\rangle _{\alpha }$ and $\left\langle \left\Vert \func{curl}%
\mathbf{F}\right\Vert \right\rangle _{\alpha }=\left\langle [\partial
_{\varrho }(\varrho F_{\bot })-\partial _{\varphi }F]/\varrho \right\rangle
_{\alpha }$, so that $2\left\langle \boldsymbol{\eth }_{\mathbf{\varrho }}%
\mathbf{F}\right\rangle _{\alpha }=\left\langle \func{div}\mathbf{F\bar{%
\varrho}}_{0}\mathbf{\varrho }_{0}+\func{curl}\mathbf{F}\times \mathbf{\bar{%
\varrho}}_{0}\mathbf{\varrho }_{0}\right\rangle _{\alpha }$, and $d\mathbf{%
\bar{S}}=d\mathbf{\hat{S}}_{\alpha }\times (\mathbf{\bar{\varrho}}_{0}%
\mathbf{\varrho }_{0})^{\alpha }$, the vector identities (\ref{36}) and (\ref%
{37}) lead to the complex generalized \textit{Stokes} integral identity%
\begin{equation}
vt\int_{\partial S^{+}}^{\circlearrowleft }\mathbf{\bar{F}}\cdot d\mathbf{%
\varrho }=2vt\int_{S^{+}}(\mathbf{\boldsymbol{\eth }_{\mathbf{\bar{\varrho}}%
}\wedge F)}\cdot d\mathbf{\bar{S}}\text{,}  \label{40}
\end{equation}%
where $\mathbf{\bar{F}}\cdot d\mathbf{\varrho }=\left\langle \mathbf{F}\circ
d\mathbf{\varrho }\right\rangle _{\alpha }\cdot (\mathbf{\bar{\varrho}}_{0}%
\mathbf{\varrho }_{0})^{\alpha }$, as well as to the complex vector integral
identity%
\begin{equation}
vt\int_{\partial S^{+}}^{\circlearrowleft }\mathbf{F}\times d\mathbf{\varrho 
}=2\langle vt\int_{G^{+}}\mathbf{\boldsymbol{\eth }_{\mathbf{\bar{\varrho}}%
}\circ F\rangle }^{\alpha }\times d\mathbf{\hat{S}}_{\alpha }\text{,}
\label{41}
\end{equation}%
where $\mathbf{F}\times d\mathbf{\varrho }=(\mathbf{\bar{\varrho}}_{0}%
\mathbf{\varrho }_{0})^{\alpha }\times \left\langle \mathbf{F}\wedge d%
\mathbf{\varrho }\right\rangle _{\alpha }$. Therefore, the generalized
complex \textit{Stokes} identity (\ref{40}) and the complex vector integral
identity (\ref{41}) appear naturally as direct consequences of the
fundamental theorem of integral calculus in $\mathbf{V}_{\mathbb{C}}^{3}$.

Accordingly, the fundamental theorem of integral calculus in $\mathbf{V}_{%
\mathbb{C}}^{3}$ can now be formulated as follows.

\begin{theorem}
Let $\gamma $ be a smooth closed spatial curve in $\mathbf{V}_{\mathbb{C}%
}^{3}$, bounding a simply connected smooth surface $S$. For an arbitrary
uniform complex vector field $\mathbf{F}$, which is regular almost
everywhere on $S$ and whose differential form $\mathbf{\bar{\omega}}_{%
\mathbf{\varrho }}=\mathbf{\bar{F}}d\mathbf{\varrho }$ is totally integrable
on $\gamma $, there holds%
\begin{equation}
vt\int_{\partial S^{+}}^{\circlearrowleft }\mathbf{\bar{\omega}}_{\mathbf{%
\varrho }}=vt\int_{S^{+}}\mathbf{\bar{\omega}}_{\mathbf{\hat{S}}}\text{,}
\label{42}
\end{equation}%
where $\mathbf{\bar{\omega}}_{\mathbf{\hat{S}}}=2\overline{\mathbf{%
\boldsymbol{\eth }}_{\mathbf{\varrho }}\mathbf{F}}d\mathbf{\hat{S}}$.
\end{theorem}

An analogous volume version is obtained by using the same approach as in 
\textit{Theorem} 5, but with the \textit{Gauss--Ostrogradsky} theorem
replacing \textit{Green's} theorem.

\begin{theorem}
Let $S$ be a closed smooth surface in the $3\mathbb{D}$ field of complex
vectors, bounding an arbitrary simply connected region $V$ in that field.
Then, for an arbitrary uniform complex vector field $\mathbf{F}$, which is
regular almost everywhere on $V$ and whose differential form $\mathbf{\omega 
}_{\mathbf{\bar{S}}}=\mathbf{F}d\mathbf{\bar{S}}$ is totally integrable on $%
S $, there holds

\begin{equation}
vt\int_{\partial V^{+}}^{\circlearrowleft }\mathbf{\omega }_{\mathbf{\bar{S}}%
}=vt\int_{V^{+}}\mathbf{\omega }_{V}\text{,}  \label{43}
\end{equation}%
where $\mathbf{\omega }_{V}=2\mathbf{\boldsymbol{\eth }_{\varrho }F}dV$.
\end{theorem}

The integral formula (\ref{43}) follows directly from the integral
identities 
\begin{equation}
vt\int_{\partial V^{+}}^{\circlearrowleft }\mathbf{F}_{\alpha }(\mathbf{\bar{%
n}}^{\alpha }\cdot d\mathbf{S)}=vt\int_{\partial V^{+}}^{\circlearrowleft
}\langle \mathbf{F}d\mathbf{S}\rangle _{\alpha }\mathbf{e}^{\alpha }=
\label{44}
\end{equation}%
\begin{equation*}
=vt\int_{\partial V^{+}}^{\circlearrowleft }\langle \mathbf{\bar{F}}\circ d%
\mathbf{S}+\mathbf{\bar{F}}\wedge d\mathbf{S}\rangle _{\alpha }\mathbf{e}%
^{\alpha }=2vt\int_{V^{+}}\langle \mathbf{\boldsymbol{\eth }_{\bar{\varrho}}}%
\circ F\mathbf{e}+\mathbf{\boldsymbol{\eth }_{\bar{\varrho}}}\wedge \hat{F}%
\mathbf{\mathbf{\mathbf{\hat{e}}}\rangle }_{\alpha }\mathbf{e}^{\alpha }dV%
\text{ and}
\end{equation*}%
\begin{equation}
vt\int_{\partial V^{+}}^{\circlearrowleft }\langle \mathbf{\bar{F}}\times 
\mathbf{n}\rangle _{\alpha }(\mathbf{\bar{n}}^{\alpha }\cdot d\mathbf{S)}%
=-vt\int_{\partial V^{+}}^{\circlearrowleft }\langle \mathbf{F}d\mathbf{\hat{%
S}}\rangle _{\alpha }\mathbf{e}^{\alpha }=  \label{45}
\end{equation}%
\begin{equation*}
=-vt\int_{\partial V^{+}}^{\circlearrowleft }\langle \mathbf{\bar{F}}\circ d%
\mathbf{\hat{S}}+\mathbf{\bar{F}}\wedge d\mathbf{\hat{S}}\rangle _{\alpha }%
\mathbf{e}^{\alpha }=2vt\int_{V^{+}}\langle \mathbf{\boldsymbol{\eth }_{\bar{%
\varrho}}}\circ \hat{F}\mathbf{\mathbf{\hat{e}}}+\mathbf{\boldsymbol{\eth }_{%
\bar{\varrho}}}\wedge F\mathbf{\mathbf{e}\rangle }_{\alpha }\mathbf{e}%
^{\alpha }dV\text{,}
\end{equation*}%
where $\mathbf{n}_{\alpha }=\langle \mathbf{e}\times \mathbf{\hat{e}\rangle }%
_{\alpha }$, $d\mathbf{S}_{\alpha }=\langle (\mathbf{e}\times d\mathbf{\hat{S%
}})\times \mathbf{\hat{e}\rangle }_{\alpha }$ and $d\mathbf{\bar{S}}_{\alpha
}=\langle d\mathbf{S}-d\mathbf{\hat{S}}\rangle _{\alpha }$. Consequently,
since $2\langle \mathbf{\boldsymbol{\eth }_{\bar{\varrho}}}\circ \mathbf{%
F\rangle }_{\alpha }\cdot \mathbf{e}^{\alpha }=\func{div}\mathbf{F}$ and $%
\mathbf{e}^{\alpha }\times 2\langle \mathbf{\boldsymbol{\eth }_{\bar{\varrho}%
}}\wedge \mathbf{F}\rangle _{\alpha }=\func{curl}\mathbf{F}$, it follows that%
\begin{equation}
vt\int_{\partial V^{+}}^{\circlearrowleft }\mathbf{F}\cdot d\mathbf{\bar{S}}%
=vt\int_{V^{+}}\func{div}\mathbf{F}dV\text{ and}  \label{46}
\end{equation}%
\begin{equation}
vt\int_{\partial V^{+}}^{\circlearrowleft }d\mathbf{S\times F}=vt\int_{V^{+}}%
\func{curl}\mathbf{F}dV\text{.}  \label{47}
\end{equation}

\begin{remark}
The fundamental integral identities established above have the same
structural form, differing only in the corresponding differential forms.
Thus, they admit the unified representation%
\begin{equation}
vt\int_{\partial \Omega ^{+}}^{\circlearrowleft }\mathbf{\omega }%
=vt\int_{\Omega ^{+}}\mathtt{d}\mathbf{\omega }\text{,}  \label{48}
\end{equation}%
where $\Omega $ is the corresponding compact set of points in a 2$\mathbb{D}$
or 3$\mathbb{D}$ field of complex vectors with boundary $\partial \Omega $
and 
\begin{equation}
\mathtt{d}\mathbf{\omega }=\lim_{\Omega \rightarrow \mathbf{\varrho }%
}\int_{\partial \Omega }^{\circlearrowleft }\mathbf{\omega }\text{.}
\label{49}
\end{equation}%
So, the integral identity $($\ref{48}$)$ serves as the common integral
identity underlying all fundamental integral theorems developed in this
paper.
\end{remark}

\section{Hermitian manifolds in $\mathbf{V}_{\mathbb{C}}^{3}$}

The developed formalism admits a natural application to \textit{Hermitian}
manifolds embedded in $\mathbf{V}_{\mathbb{C}}^{3}$. Let $\langle \mathbf{%
\varrho }\rangle _{\alpha }=\langle \varrho (x^{i})\mathbf{e}+\hat{\varrho}%
(x^{i})\mathbf{\hat{e}}\rangle _{\alpha }$ be a coordinate transformation of
the position vector in the component fields $\left\langle \mathbf{V}_{%
\mathbb{C}}\right\rangle _{\alpha }$. Here, the index $i$ runs from $1$ to $%
1 $ or $2$. If $\hat{\varrho}_{1}=\varrho _{2}$, $\hat{\varrho}_{2}=\varrho
_{3}$ and $\hat{\varrho}_{3}=\varrho _{1}$, the component vectors%
\begin{equation}
\langle \mathbf{g}_{i}\rangle _{\alpha }=\langle \partial _{i}\mathbf{%
\varrho }\rangle _{\alpha }=\langle \partial _{x^{i}}\varrho (x^{i})\mathbf{e%
}+\partial _{x^{i}}\hat{\varrho}(x^{i})\mathbf{\hat{e}}\rangle _{\alpha }
\label{50}
\end{equation}%
induce the covariant basis vectors $\mathbf{g}_{i}=\langle \mathbf{g}%
_{i}\rangle _{\alpha }\mathbf{e}^{\alpha }$ of the compact \textit{Hermitian}
manifold embedded in the ambient 3$\mathbb{D}$ field of complex vectors $%
\mathbf{V}_{\mathbb{C}}^{3}$. Accordingly, on the one hand%
\begin{equation}
\langle \mathbf{g}_{i}\circ \mathbf{g}_{j}\rangle _{\alpha }=\langle
(\partial _{i}\varrho \partial _{j}\varrho +\partial _{i}\hat{\varrho}%
\partial _{j}\hat{\varrho})\mathbf{e}\rangle _{\alpha }=\langle g_{ij}%
\mathbf{e}\rangle _{\alpha }\text{ and}  \label{51}
\end{equation}%
\begin{equation*}
\langle \mathbf{g}_{i}\wedge \mathbf{g}_{j}\rangle _{\alpha }=\left\vert 
\begin{array}{cc}
\partial _{i}\varrho & \partial _{i}\hat{\varrho} \\ 
\partial _{j}\varrho & \partial _{j}\hat{\varrho}%
\end{array}%
\right\vert \mathbf{\hat{e}}=\langle (\partial _{i}\varrho \partial _{j}\hat{%
\varrho}-\partial _{i}\hat{\varrho}\partial _{j}\varrho )\mathbf{\hat{e}}%
\rangle _{\alpha }=\langle \hat{g}_{ij}\mathbf{\hat{e}}\rangle _{\alpha }%
\text{.}
\end{equation*}%
On the other hand,%
\begin{equation}
\langle \mathbf{\hat{e}g}_{i}\wedge \mathbf{g}_{j}\rangle _{\alpha
}=\left\vert 
\begin{array}{cc}
-\partial _{i}\hat{\varrho} & \partial _{i}\varrho \\ 
\partial _{j}\varrho & \partial _{j}\hat{\varrho}%
\end{array}%
\right\vert \mathbf{\hat{e}}=\langle -(\partial _{i}\varrho \partial
_{j}\varrho +\partial _{i}\hat{\varrho}\partial _{j}\hat{\varrho})\mathbf{%
\hat{e}}\rangle _{\alpha }=\langle -g_{ij}\mathbf{\hat{e}}\rangle _{\alpha }%
\text{ and}  \label{52}
\end{equation}%
\begin{equation*}
\langle \mathbf{\hat{e}g}_{i}\circ \mathbf{g}_{j}\rangle _{\alpha }=\langle
(\partial _{i}\varrho \partial _{j}\hat{\varrho}-\partial _{i}\hat{\varrho}%
\partial _{j}\varrho )\mathbf{e}\rangle _{\alpha }=\langle -(\mathbf{g}%
_{i}\wedge \mathbf{g}_{j})\mathbf{\hat{e}}\rangle _{\alpha }=\langle \hat{g}%
_{ij}\mathbf{e}\rangle _{\alpha }\text{.}
\end{equation*}%
Consequently, $\langle \hat{g}_{ij}\rangle _{\alpha }=\langle -\hat{g}%
_{ji}\rangle _{\alpha }$. The contravariant component basis vectors $\langle 
\mathbf{g}^{i}\rangle _{\alpha }$, which are dual to covariant component
basis vectors $\langle \mathbf{g}_{i}\rangle _{\alpha }$, are defined by%
\begin{equation}
\langle \mathbf{g}^{i}\rangle _{\alpha }=\langle \hat{g}^{ij}\mathbf{\hat{e}g%
}_{j}\rangle _{\alpha }\text{,}  \label{53}
\end{equation}%
where $\hat{g}^{ij}$ is a square antisymmetric matrix,\ such that%
\begin{equation}
\hat{g}^{ij}=\left[ 
\begin{array}{cc}
0 & \hat{g}^{12} \\ 
\hat{g}^{21} & 0%
\end{array}%
\right] =\left[ 
\begin{array}{cc}
0 & -\hat{g}_{12}^{-1} \\ 
\hat{g}_{12}^{-1} & 0%
\end{array}%
\right] \text{ and}  \label{54}
\end{equation}%
\begin{equation*}
\langle \mathbf{g}^{i}\circ \mathbf{g}_{j}\rangle _{\alpha }=\langle \hat{g}%
^{ik}\mathbf{\hat{e}g}_{k}\circ \mathbf{g}_{j}\rangle _{\alpha }=\langle 
\hat{g}^{ik}\hat{g}_{kj}\mathbf{e}\rangle _{\alpha }=\langle \delta _{j}^{i}%
\mathbf{e}\rangle _{\alpha }\text{,}
\end{equation*}%
where $\delta _{j}^{i}$ is the \textit{Kronecker} delta (the identity
matrix). Here and further in the text, the \textit{Einstein} summation
convention applies over repeated \textit{Latin} indices. By (\ref{52}),%
\begin{equation}
\langle \mathbf{g}^{i}\wedge \mathbf{g}_{j}\rangle _{\alpha }=\langle \hat{g}%
^{ik}\mathbf{\hat{e}g}_{k}\wedge \mathbf{g}_{j}\rangle _{\alpha }=\langle -%
\hat{g}^{ik}g_{kj}\mathbf{\hat{e}}\rangle _{i}=\langle \hat{\delta}_{j}^{i}%
\mathbf{\hat{e}}\rangle _{\alpha }\text{,}  \label{55}
\end{equation}%
so that%
\begin{equation}
\mathbf{g}_{ij}=\mathbf{\bar{g}}_{i}\mathbf{g}_{j}=\mathbf{g}_{i}\circ 
\mathbf{g}_{j}+\mathbf{g}_{i}\wedge \mathbf{g}_{j}=\langle \mathbf{g}%
_{i}\circ \mathbf{\mathbf{g}}_{j}+\mathbf{\mathbf{g}}_{i}\wedge \mathbf{g}%
_{j}\rangle _{\alpha }\mathbf{e}^{\alpha }=  \label{56}
\end{equation}%
\begin{equation*}
=\langle g_{ij}\mathbf{e}+\hat{g}_{ij}\mathbf{\hat{e}}\rangle _{\alpha }%
\mathbf{e}^{\alpha }=\langle \mathbf{\bar{g}}_{i}\mathbf{g}_{j}\rangle
_{\alpha }\mathbf{e}^{\alpha }=\langle \mathbf{g}_{ij}\rangle _{\alpha }%
\mathbf{e}^{\alpha }\text{ and}
\end{equation*}%
\begin{equation*}
\mathbf{\delta }_{j}^{i}=\mathbf{\bar{g}}^{i}\mathbf{g}_{j}=\mathbf{g}%
^{i}\circ \mathbf{g}_{j}+\mathbf{g}^{i}\wedge \mathbf{g}_{j}=\langle \mathbf{%
g}^{i}\circ \mathbf{\mathbf{g}}_{j}+\mathbf{\mathbf{g}}^{i}\wedge \mathbf{g}%
_{j}\rangle _{\alpha }\mathbf{e}^{\alpha }=
\end{equation*}%
\begin{equation*}
=\langle \delta _{j}^{i}\mathbf{e}+\hat{\delta}_{j}^{i}\mathbf{\hat{e}}%
\rangle _{\alpha }\mathbf{e}^{\alpha }=\langle \mathbf{\bar{g}}^{i}\mathbf{g}%
_{j}\rangle _{\alpha }\mathbf{e}^{\alpha }=\langle \mathbf{\delta }%
_{j}^{i}\rangle _{\alpha }\mathbf{e}^{\alpha }\text{.}
\end{equation*}

Since%
\begin{equation}
\hat{g}^{ik}g_{kj}=\left[ 
\begin{array}{cc}
0 & \hat{g}^{12} \\ 
-\hat{g}^{12} & 0%
\end{array}%
\right] \left[ 
\begin{array}{cc}
g_{11} & g_{12} \\ 
g_{12} & g_{22}%
\end{array}%
\right] =\left[ 
\begin{array}{cc}
\hat{g}^{12}g_{12} & \hat{g}^{12}g_{22} \\ 
-\hat{g}^{12}g_{11} & -\hat{g}^{12}g_{12}%
\end{array}%
\right] =-\hat{\delta}_{j}^{i}\text{ and}  \label{57}
\end{equation}%
\begin{equation*}
\hat{g}^{ik}\hat{g}_{kj}=\left[ 
\begin{array}{cc}
0 & \hat{g}^{12} \\ 
-\hat{g}^{12} & 0%
\end{array}%
\right] \left[ 
\begin{array}{cc}
0 & \hat{g}_{12} \\ 
-\hat{g}_{12} & 0%
\end{array}%
\right] =\left[ 
\begin{array}{cc}
-\hat{g}^{12}\hat{g}_{12} & 0 \\ 
0 & -\hat{g}^{12}\hat{g}_{12}%
\end{array}%
\right] =\delta _{j}^{i}\text{,}
\end{equation*}%
it follows that $\hat{\delta}_{i}^{i}=0$. Therefore, $\langle \mathbf{\bar{g}%
}^{i}\mathbf{g}_{i}\rangle _{\alpha }=\langle \mathbf{\delta }%
_{i}^{i}\rangle _{\alpha }=\langle \delta _{i}^{i}\mathbf{e}\rangle _{\alpha
}=\langle 2\mathbf{e}\rangle _{\alpha }$ and%
\begin{equation}
\langle \mathbf{\bar{g}}^{i}\rangle _{\alpha }=\langle \frac{\mathbf{\bar{g}}%
^{i}\mathbf{g}_{i}}{\mathbf{g}_{i}}\rangle _{\alpha }=2\langle \frac{1}{%
\mathbf{g}_{i}}\rangle _{\alpha }\text{.}  \label{58}
\end{equation}%
By (\ref{57}), $\hat{\delta}_{k}^{i}\hat{\delta}_{j}^{k}=g\hat{g}\delta
_{j}^{i}$, where $g=\left\vert g_{jk}\right\vert =g_{11}g_{22}-(g_{12})^{2}$
and $\hat{g}=\left\vert \hat{g}^{ik}\right\vert =-(\hat{g}^{12})^{2}$, so
that $g\hat{g}=-\left\vert \hat{\delta}_{j}^{i}\right\vert =-1$, since 
\begin{equation}
2\langle \mathbf{\delta }_{j}^{i}\rangle _{\alpha }=\langle \mathbf{\delta }%
_{k}^{i}\mathbf{\delta }_{j}^{k}\rangle _{\alpha }=\langle (\delta _{k}^{i}%
\mathbf{e}+\hat{\delta}_{k}^{i}\mathbf{\hat{e})}(\delta _{j}^{k}\mathbf{e}+%
\hat{\delta}_{j}^{k}\mathbf{\hat{e})}\rangle _{\alpha }=  \label{59}
\end{equation}%
\begin{equation*}
=\langle (\delta _{k}^{i}\delta _{j}^{k}-\hat{\delta}_{k}^{i}\hat{\delta}%
_{j}^{k})\mathbf{e}+(\delta _{k}^{i}\hat{\delta}_{j}^{k}+\hat{\delta}%
_{k}^{i}\delta _{j}^{k}\mathbf{)\hat{e}}\rangle _{\alpha }=\langle (\delta
_{j}^{i}-\hat{\delta}_{k}^{i}\hat{\delta}_{j}^{k})\mathbf{e}+2\hat{\delta}%
_{j}^{i}\mathbf{\hat{e}}\rangle _{\alpha }\text{. }
\end{equation*}%
Consequently, $\langle \mathbf{\bar{\delta}}_{k}^{i}\mathbf{\delta }%
_{j}^{k}\rangle _{\alpha }=\mathbf{0}$ implying that $\langle \mathbf{g}^{i}%
\mathbf{g}_{i}\rangle _{\alpha }=\mathbf{0}$. Furthermore,%
\begin{equation}
\langle \mathbf{g}^{ij}\rangle _{\alpha }=\langle \mathbf{\bar{g}}^{i}%
\mathbf{g}^{j}\rangle _{\alpha }=\langle -\hat{g}^{ik}\mathbf{\hat{e}\bar{g}}%
_{k}\mathbf{g}^{j}\rangle _{\alpha }=\langle \hat{g}^{ik}\hat{g}^{jl}\mathbf{%
\bar{g}}_{k}\mathbf{g}_{l}\rangle _{\alpha }=  \label{60}
\end{equation}%
\begin{equation*}
=\langle \hat{g}^{ik}\hat{g}^{jl}(g_{kl}\mathbf{e}+\hat{g}_{kl}\mathbf{\hat{e%
})}\rangle _{\alpha }=\langle \hat{g}^{jl}(-\hat{\delta}_{l}^{i}\mathbf{e}%
+\delta _{l}^{i}\mathbf{\hat{e}})\rangle _{\alpha }=\langle \hat{g}^{jl}%
\mathbf{\delta }_{l}^{i}\mathbf{\hat{e}}\rangle _{\alpha }=\langle g^{ij}%
\mathbf{e}-\hat{g}^{ij}\mathbf{\hat{e}}\rangle _{\alpha }\text{,}
\end{equation*}%
where $\langle \mathbf{g}^{i}\circ \mathbf{g}^{j}\rangle _{\alpha }=\langle
g^{ij}\mathbf{e}\rangle _{\alpha }$ and $\langle \mathbf{g}^{i}\wedge 
\mathbf{g}^{j}\rangle _{\alpha }=-\langle \hat{g}^{ij}\mathbf{\hat{e}}%
\rangle _{\alpha }$.

The component \textit{Wirtinger} operators (\ref{13}) are given by%
\begin{equation}
2\langle \boldsymbol{\eth }_{\mathbf{\varrho }}\rangle _{\alpha }=2\langle
(\partial _{i}\mathbf{\varrho )}^{-1}\partial _{i}\rangle _{\alpha }=\langle 
\mathbf{\bar{g}}^{i}\partial _{i}\rangle _{\alpha }\text{ and}  \label{61}
\end{equation}%
\begin{equation*}
\langle \mathbf{g}^{i}\partial _{i}\rangle _{\alpha }=2\langle \boldsymbol{%
\eth }_{\mathbf{\bar{\varrho}}}\rangle _{\alpha }=2\langle \partial _{%
\mathbf{\bar{\varrho}}}\varrho \partial _{\varrho }+\partial _{\mathbf{\bar{%
\varrho}}}\hat{\varrho}\partial _{\hat{\varrho}}\rangle _{\alpha }=2\langle 
\mathbf{e}\partial _{\varrho }+\mathbf{\hat{e}}\partial _{\hat{\varrho}%
}\rangle _{\alpha }\text{,}
\end{equation*}%
where $\partial _{\mathbf{\bar{\varrho}}}\varrho =(\partial _{\varrho }%
\mathbf{\bar{\varrho}})^{-1}=\mathbf{e}^{-1}=\mathbf{e}$ and $\partial _{%
\mathbf{\bar{\varrho}}}\hat{\varrho}=(\partial _{\hat{\varrho}}\mathbf{\bar{%
\varrho}})^{-1}=\mathbf{-\hat{e}}^{-1}=\mathbf{\hat{e}}$, so that%
\begin{equation}
2\langle \boldsymbol{\eth }_{i}\rangle _{\alpha }=2\langle \partial _{i}%
\mathbf{\varrho }\boldsymbol{\eth }_{\mathbf{\varrho }}\rangle _{\alpha
}=\langle \mathbf{g}_{i}\mathbf{\bar{g}}^{j}\partial _{j}\rangle _{\alpha
}=\langle \mathbf{\delta }_{i}^{j}\partial _{j}\rangle _{\alpha }=\langle 
\mathbf{e}\partial _{i}+\mathbf{\hat{e}}\hat{\delta}_{i}^{j}\partial
_{j}\rangle _{\alpha }\text{ and}  \label{62}
\end{equation}%
\begin{equation*}
2\langle \boldsymbol{\eth }_{i}\mathbf{\varrho }\rangle _{\alpha }=2\langle 
\mathbf{g}_{i}\rangle _{\alpha }=\langle \mathbf{eg}_{i}+\mathbf{\hat{e}}%
\hat{\delta}_{i}^{j}\mathbf{g}_{j}\rangle _{\alpha }\text{.}
\end{equation*}%
Consequently, $\langle \mathbf{\hat{e}}\hat{\delta}_{i}^{j}\mathbf{g}%
_{j}\rangle _{\alpha }=\langle \mathbf{g}_{i}\rangle _{\alpha }$ and $%
\langle \hat{\delta}_{i}^{j}\mathbf{g}^{i}\rangle _{\alpha }=\langle \mathbf{%
\hat{e}g}^{j}\rangle _{\alpha }$. Furthermore, 
\begin{equation}
\langle \mathbf{d}_{t}\rangle _{\alpha }=2\langle \mathbf{\nu }\circ 
\boldsymbol{\eth }_{\mathbf{\bar{\varrho}}}\rangle _{\alpha }=\langle 
\mathbf{\nu }\boldsymbol{\eth }_{\mathbf{\varrho }}+\mathbf{\bar{\nu}}%
\boldsymbol{\eth }_{\mathbf{\bar{\varrho}}}\rangle _{\alpha }=\langle \nu
^{i}(\boldsymbol{\eth }_{i}+\boldsymbol{\bar{\eth }}_{i})\rangle _{\alpha }%
\text{,}  \label{63}
\end{equation}%
where $\langle \mathbf{\nu }\rangle _{\alpha }=\langle d_{t}\mathbf{\varrho }%
\rangle _{\alpha }=\langle v^{i}\mathbf{g}_{i}\rangle _{\alpha }$, is the
time differential operator. Hence,

\begin{equation}
\langle \mathbf{d}_{t^{2}}^{2}\rangle _{\alpha }=\langle \mathbf{d}_{t}%
\mathbf{d}_{t}\rangle _{\alpha }=\nu ^{j}[\mathbf{e}(\partial _{j}\nu
^{i}\partial _{i}+\nu ^{i}\partial _{ij}^{2})+\nu ^{i}\partial _{j}(\mathbf{%
\delta }_{i}^{k}+\mathbf{\bar{\delta}}_{i}^{k})\partial _{k}]=\mathbf{e}\nu
^{j}\partial _{j}(\nu ^{i}\partial _{i})\text{.}  \label{64}
\end{equation}

On the other hand, $\langle \hat{\delta}_{j}^{i}\mathbf{\hat{e}}\rangle
_{\alpha }\times \mathbf{e}^{\alpha }=\langle \mathbf{g}^{i}\wedge \mathbf{g}%
_{j}\rangle _{\alpha }\times \mathbf{e}^{\alpha }=\mathbf{g}^{i}\times 
\mathbf{g}_{j}=\gamma _{j}^{i}\mathbf{n}$, where $\mathbf{n}$ is the normal
unit vector field on the \textit{Hermitian} manifold. Since $\partial _{i}%
\mathbf{n}$ are vectors in the tangent vector space, it follows that%
\begin{equation}
\partial _{i}\mathbf{n}=-t_{i}^{j}\langle \partial _{j}\hat{\varrho}\mathbf{%
\hat{e}}\rangle _{\alpha }\mathbf{e}^{\alpha }\text{ and}  \label{65}
\end{equation}%
\begin{equation*}
\partial _{j}(\langle \partial _{i}\hat{\varrho}\mathbf{\hat{e}}\rangle
_{\alpha }\mathbf{e}^{\alpha })=\Gamma _{ij}^{k}(\langle \partial _{k}\hat{%
\varrho}\mathbf{\hat{e}}\rangle _{\alpha }\mathbf{e}^{\alpha })+t_{ij}%
\mathbf{n}\text{,}
\end{equation*}%
where $\Gamma _{ij}^{k}$ are the \textit{Christoffel} symbols and $%
t_{ij}=\partial _{j}(\langle \partial _{i}\hat{\varrho}\mathbf{\hat{e}}%
\rangle _{\alpha }\mathbf{e}^{\alpha })\cdot \mathbf{n=-}\langle \partial
_{i}\hat{\varrho}\mathbf{\hat{e}}\rangle _{\alpha }\mathbf{e}^{\alpha }\cdot
\partial _{j}\mathbf{n}$. Accordingly,%
\begin{equation}
\langle \partial _{ij}^{2}(\langle \partial _{k}\hat{\varrho}\mathbf{\hat{e}}%
\rangle _{\alpha }\mathbf{e}^{\alpha })-\partial _{j}[\Gamma
_{ik}^{l}(\langle \partial _{l}\hat{\varrho}\mathbf{\hat{e}}\rangle _{\alpha
}\mathbf{e}^{\alpha })]+t_{ik}t_{j}^{l}(\langle \partial _{l}\hat{\varrho}%
\mathbf{\hat{e}}\rangle _{\alpha }\mathbf{e}^{\alpha })=\partial _{j}t_{ik}%
\mathbf{n}\text{ and}  \label{66}
\end{equation}%
\begin{equation*}
\langle \partial _{ji}^{2}(\langle \partial _{k}\hat{\varrho}\mathbf{\hat{e}}%
\rangle _{\alpha }\mathbf{e}^{\alpha })-\partial _{i}[\Gamma
_{jk}^{l}(\langle \partial _{l}\hat{\varrho}\mathbf{\hat{e}}\rangle _{\alpha
}\mathbf{e}^{\alpha })]+t_{jk}t_{i}^{l}(\langle \partial _{l}\hat{\varrho}%
\mathbf{\hat{e}}\rangle _{\alpha }\mathbf{e}^{\alpha })=\partial _{i}t_{jk}%
\mathbf{n}\text{.}
\end{equation*}%
Since $\partial _{ij}^{2}(\langle \partial _{k}\hat{\varrho}\mathbf{\hat{e}}%
\rangle _{\alpha }\mathbf{e}^{\alpha })-\partial _{ji}^{2}(\langle \partial
_{k}\hat{\varrho}\mathbf{\hat{e}}\rangle _{\alpha }\mathbf{e}^{\alpha })=%
\mathbf{0}$, it follows that 
\begin{equation}
\partial _{i}[\Gamma _{jk}^{l}(\langle \partial _{l}\hat{\varrho}\mathbf{%
\hat{e}}\rangle _{\alpha }\mathbf{e}^{\alpha })]-\partial _{j}[\Gamma
_{ik}^{l}(\langle \partial _{l}\hat{\varrho}\mathbf{\hat{e}}\rangle _{\alpha
}\mathbf{e}^{\alpha })]+  \label{67}
\end{equation}%
\begin{equation*}
+(t_{ik}t_{j}^{l}-t_{jk}t_{i}^{l})(\langle \partial _{l}\hat{\varrho}\mathbf{%
\hat{e}}\rangle _{\alpha }\mathbf{e}^{\alpha })=(\partial
_{j}t_{ik}-\partial _{i}t_{jk})\mathbf{n}\text{ and}
\end{equation*}%
\begin{equation*}
(\mathcal{R}_{kij}^{l}+t_{ki}t_{j}^{l}-t_{kj}t_{i}^{l})(\langle \partial _{l}%
\hat{\varrho}\mathbf{\hat{e}}\rangle _{\alpha }\mathbf{e}^{\alpha })=\mathbf{%
(\bigtriangledown }_{j}t_{ik}-\mathbf{\bigtriangledown }_{i}t_{jk})\mathbf{n}%
\text{,}
\end{equation*}%
where%
\begin{equation}
\mathcal{R}_{kij}^{l}=\partial _{i}\Gamma _{jk}^{l}-\partial _{j}\Gamma
_{ik}^{l}+\Gamma _{im}^{l}\Gamma _{jk}^{m}-\Gamma _{jm}^{l}\Gamma _{ik}^{m}%
\text{ and}  \label{68}
\end{equation}%
\begin{equation*}
\mathbf{\bigtriangledown }_{i}t_{jk}=\partial _{i}t_{jk}-\Gamma
_{ik}^{l}t_{lj}-\Gamma _{ij}^{l}t_{lk}\text{\textbf{,}}
\end{equation*}%
are the \textit{Riemann} curvature tensor and covariant derivative of the
tensor $t_{jk}$ of the second fundamental form, respectively. The second
vector identity in (\ref{67}) leads to both the \textit{Gauss} and \textit{%
Codazzi}-\textit{Mainardi} equations.

Further, let $\mathbf{N}$ denotes a complex vector field in $\mathbf{V}_{%
\mathbb{C}}^{3}$, such that $\partial _{i}\mathbf{N=}-\langle t_{i}^{j}%
\mathbf{g}_{j}\rangle _{\alpha }\mathbf{e}^{\alpha }$. By (\ref{61})%
\begin{equation}
\boldsymbol{\eth }_{\mathbf{\varrho }}\mathbf{N=\bar{g}}^{i}\partial _{i}%
\mathbf{N}=-\langle t_{i}^{j}\mathbf{g}_{j}\mathbf{\bar{g}}^{i}\rangle
_{\alpha }\mathbf{e}^{\alpha }=-\langle t_{i}^{j}\mathbf{\delta }%
_{j}^{i}\rangle _{\alpha }\mathbf{e}^{\alpha }=-t_{i}^{j}\mathbf{\delta }%
_{j}^{i}\text{,}  \label{69}
\end{equation}%
Based on the results of \textit{Theorem} 2, it follows that%
\begin{equation}
vt\int_{\partial S^{+}}^{\circlearrowleft }\mathbf{N}d\mathbf{\bar{\varrho}}%
=2vt\int_{S^{+}}t_{i}^{j}\mathbf{\delta }_{j}^{i}d\mathbf{\hat{S}}\text{.}
\label{70}
\end{equation}

The differential and integral structures developed above naturally extend to
the vector \textit{Gauss} equation for \textit{Hermitian} manifolds, which
is derived below.

\subsection{The vector \textit{Gauss} equation for \textit{Hermitian}
manifolds}

The following derivation shows that the classical \textit{Gauss} equation
admits a natural vector formulation in the present algebraic setting. By the
first vector equality in (\ref{62}), 
\begin{equation}
4\langle \boldsymbol{\eth }_{ij}^{2}\rangle _{\alpha }=4\langle \boldsymbol{%
\eth }_{j}\boldsymbol{\eth }_{i}\rangle _{\alpha }=\langle (\mathbf{e}%
\partial _{j}+\mathbf{\hat{e}}\hat{\delta}_{j}^{k}\partial _{k})(\mathbf{e}%
\partial _{i}+\mathbf{\hat{e}}\hat{\delta}_{i}^{l}\partial _{l})\rangle
_{\alpha }=  \label{71}
\end{equation}%
\begin{equation*}
=\mathbf{e(}\partial _{ij}^{2}-\hat{\delta}_{j}^{k}\partial _{k}\hat{\delta}%
_{i}^{l}\partial _{l}-\hat{\delta}_{j}^{k}\hat{\delta}_{i}^{l}\partial
_{lk}^{2})+\mathbf{\hat{e}(}\partial _{j}\hat{\delta}_{i}^{l}\partial _{l}+%
\hat{\delta}_{i}^{l}\partial _{lj}^{2}+\hat{\delta}_{j}^{k}\partial
_{ik}^{2})\text{ and}
\end{equation*}%
\begin{equation*}
4\langle \boldsymbol{\eth }_{ji}^{2}\rangle _{\alpha }=4\langle \boldsymbol{%
\eth }_{i}\boldsymbol{\eth }_{j}\rangle _{\alpha }=\langle (\mathbf{e}%
\partial _{i}+\mathbf{\hat{e}}\hat{\delta}_{i}^{l}\partial _{l})(\mathbf{e}%
\partial _{j}+\mathbf{\hat{e}}\hat{\delta}_{j}^{k}\partial _{k})\rangle
_{\alpha }=
\end{equation*}%
\begin{equation*}
=\langle \mathbf{e(}\partial _{ji}^{2}-\hat{\delta}_{i}^{l}\partial _{l}\hat{%
\delta}_{j}^{k}\partial _{k}-\hat{\delta}_{i}^{l}\hat{\delta}%
_{j}^{k}\partial _{kl}^{2})+\mathbf{\hat{e}(}\partial _{i}\hat{\delta}%
_{j}^{k}\partial _{k}+\hat{\delta}_{j}^{k}\partial _{ki}^{2}+\hat{\delta}%
_{i}^{l}\partial _{jl}^{2})\rangle _{\alpha }\text{,}
\end{equation*}%
so that%
\begin{equation}
4\langle \boldsymbol{\eth }_{ij}^{2}-\boldsymbol{\eth }_{ji}^{2}\rangle
_{\alpha }=\langle \lbrack \mathbf{e(}\hat{\delta}_{i}^{l}\partial _{l}\hat{%
\delta}_{j}^{k}-\hat{\delta}_{j}^{l}\partial _{l}\hat{\delta}_{i}^{k})+%
\mathbf{\hat{e}(}\partial _{j}\hat{\delta}_{i}^{k}-\partial _{i}\hat{\delta}%
_{j}^{k})]\partial _{k}\rangle _{\alpha }=  \label{72}
\end{equation}%
\begin{equation*}
=2\langle (\boldsymbol{\eth }_{j}\mathbf{\delta }_{i}^{k}-\boldsymbol{\eth }%
_{i}\mathbf{\delta }_{j}^{k}\mathbf{)}\partial _{k}\rangle _{\alpha
}=2\langle (\boldsymbol{\eth }_{j}\mathbf{\delta }_{i}^{k}-\boldsymbol{\eth }%
_{i}\mathbf{\delta }_{j}^{k}\mathbf{)(\mathbf{g}}_{k}\boldsymbol{\eth }_{%
\mathbf{\varrho }}+\mathbf{\mathbf{\bar{g}}}_{k}\boldsymbol{\eth }_{\mathbf{%
\bar{\varrho}}})\rangle _{\alpha }\text{,}
\end{equation*}%
since%
\begin{equation}
2\langle \boldsymbol{\eth }_{j}\mathbf{\delta }_{i}^{k}\rangle _{\alpha
}=\langle (\mathbf{e}\partial _{j}+\mathbf{\hat{e}}\hat{\delta}%
_{j}^{l}\partial _{l})(\mathbf{e}\delta _{i}^{k}+\mathbf{\hat{e}}\hat{\delta}%
_{i}^{k})\rangle _{\alpha }=-\langle \mathbf{e}\hat{\delta}_{j}^{l}\partial
_{l}\hat{\delta}_{i}^{k}-\mathbf{\hat{e}}\partial _{j}\hat{\delta}%
_{i}^{k}\rangle _{\alpha }\text{.}  \label{73}
\end{equation}%
Therefore, taking into account that $\langle \mathbf{g}^{j}\mathbf{g}%
_{j}\rangle _{\alpha }=\mathbf{0}$ and $2\langle \boldsymbol{\eth }%
_{i}\rangle _{\alpha }=2\langle \mathbf{g}_{i}\boldsymbol{\eth }_{\mathbf{%
\varrho }}\rangle _{\alpha }=\langle \mathbf{\delta }_{i}^{k}\partial
_{k}\rangle _{\alpha }$, it follows that 
\begin{equation}
\langle \mathbf{g}^{ij}(\boldsymbol{\eth }_{ij}^{2}-\boldsymbol{\eth }%
_{ji}^{2})\rangle _{\alpha }=-\langle \mathbf{g}^{i}\boldsymbol{\eth }_{%
\mathbf{\varrho }}\mathbf{\delta }_{i}^{k}\partial _{k}\rangle _{\alpha
}=\langle \boldsymbol{\eth }_{\mathbf{\varrho }}\mathbf{g}^{i}\mathbf{\delta 
}_{i}^{k}\partial _{k}\rangle _{\alpha }=2\langle \boldsymbol{\eth }_{%
\mathbf{\varrho }}\mathbf{g}^{i}\boldsymbol{\eth }_{i}\rangle _{\alpha }%
\text{.}  \label{74}
\end{equation}%
On the other hand,%
\begin{equation}
\langle \mathbf{0}\rangle _{\alpha }=\langle \mathbf{g}_{i}\mathbf{g}_{j}(%
\boldsymbol{\eth }_{\mathbf{\varrho }^{2}}^{2}-\boldsymbol{\eth }_{\mathbf{%
\varrho }^{2}}^{2})\rangle _{\alpha }=\langle \boldsymbol{\eth }_{ij}^{2}-%
\boldsymbol{\eth }_{ji}^{2}-(\boldsymbol{\eth }_{i}\mathbf{g}_{j}-%
\boldsymbol{\eth }_{j}\mathbf{g}_{i})\boldsymbol{\eth }_{\mathbf{\varrho }%
}\rangle _{\alpha }\text{,}  \label{75}
\end{equation}%
so that%
\begin{equation}
\langle \mathbf{\bar{g}}^{k}(\boldsymbol{\eth }_{ij}^{2}-\boldsymbol{\eth }%
_{ji}^{2})\rangle _{\alpha }=\langle (\mathbf{\bar{g}}^{k}\boldsymbol{\eth }%
_{i}\mathbf{g}_{j}-\mathbf{\bar{g}}^{k}\boldsymbol{\eth }_{j}\mathbf{g}_{i})%
\boldsymbol{\eth }_{\mathbf{\varrho }}\rangle _{\alpha }=\langle (\mathbf{%
\Gamma }_{ij}^{k}-\mathbf{\Gamma }_{ji}^{k})\boldsymbol{\eth }_{\mathbf{%
\varrho }}\rangle _{\alpha }\text{,}  \label{76}
\end{equation}%
where $\langle \mathbf{\Gamma }_{ij}^{k}\rangle _{\alpha }=$ $\langle 
\mathbf{\bar{g}}^{k}\boldsymbol{\eth }_{i}\mathbf{g}_{j}\rangle _{\alpha }$.
This implies that%
\begin{equation}
2\langle (\boldsymbol{\eth }_{ij}^{2}-\boldsymbol{\eth }_{ji}^{2})\rangle
_{\alpha }=\langle (\mathbf{\Gamma }_{ij}^{k}-\mathbf{\Gamma }_{ji}^{k})%
\boldsymbol{\eth }_{k}\rangle _{\alpha }\text{ and}  \label{77}
\end{equation}%
\begin{equation*}
2\langle \mathbf{\bar{g}}^{l}(\boldsymbol{\eth }_{ij}^{2}\mathbf{g}_{k}-%
\boldsymbol{\eth }_{ji}^{2}\mathbf{g}_{k})\rangle _{\alpha }=\langle \mathbf{%
\bar{g}}^{l}\boldsymbol{\eth }_{i}\mathbf{g}_{j}\mathbf{\bar{g}}^{r}%
\boldsymbol{\eth }_{r}\mathbf{g}_{k}-\mathbf{\bar{g}}^{l}\boldsymbol{\eth }%
_{j}\mathbf{g}_{i}\mathbf{\bar{g}}^{r}\boldsymbol{\eth }_{r}\mathbf{g}%
_{k}\rangle _{\alpha }=
\end{equation*}%
\begin{equation*}
=\langle \mathbf{\bar{g}}^{l}\boldsymbol{\eth }_{r}\mathbf{g}_{j}\mathbf{%
\bar{g}}^{r}\boldsymbol{\eth }_{i}\mathbf{g}_{k}-\mathbf{\bar{g}}^{l}%
\boldsymbol{\eth }_{r}\mathbf{g}_{i}\mathbf{\bar{g}}^{r}\boldsymbol{\eth }%
_{j}\mathbf{g}_{k}\rangle _{\alpha }=\langle \mathbf{\Gamma }_{jr}^{l}%
\mathbf{\Gamma }_{ki}^{r}-\mathbf{\Gamma }_{ir}^{l}\mathbf{\Gamma }%
_{kj}^{r}\rangle _{\alpha }\text{.}
\end{equation*}%
Based on the preceding vector equality and (\ref{72}), one obtains 
\begin{equation}
4\langle (\boldsymbol{\eth }_{ij}^{2}\mathbf{g}_{k}-\boldsymbol{\eth }%
_{ji}^{2}\mathbf{g}_{k})\rangle _{\alpha }=\langle (\mathbf{\Gamma }_{jr}^{l}%
\mathbf{\Gamma }_{ki}^{r}-\mathbf{\Gamma }_{ir}^{l}\mathbf{\Gamma }_{kj}^{r})%
\mathbf{g}_{l}\rangle _{\alpha }=2\langle (\boldsymbol{\eth }_{j}\mathbf{%
\delta }_{i}^{l}-\boldsymbol{\eth }_{i}\mathbf{\delta }_{j}^{l}\mathbf{)}%
\partial _{l}\mathbf{g}_{k}\rangle _{\alpha }\text{.}  \label{78}
\end{equation}%
Since%
\begin{equation}
\langle \mathbf{\Gamma }_{rj}^{l}\mathbf{\Gamma }_{ki}^{r}-\mathbf{\Gamma }%
_{ri}^{l}\mathbf{\Gamma }_{kj}^{r}\rangle _{\alpha }=\langle \mathbf{\delta }%
_{j}^{l}\mathbf{\delta }_{i}^{r}\boldsymbol{\eth }_{\mathbf{\varrho }}%
\mathbf{g}_{r}\boldsymbol{\eth }_{\mathbf{\varrho }}\mathbf{g}_{k}-\mathbf{%
\delta }_{i}^{l}\mathbf{\delta }_{j}^{r}\boldsymbol{\eth }_{\mathbf{\varrho }%
}\mathbf{g}_{r}\boldsymbol{\eth }_{\mathbf{\varrho }}\mathbf{g}_{k}\rangle
_{\alpha }=\langle \mathbf{0}\rangle _{\alpha }\text{ and}  \label{79}
\end{equation}%
\begin{equation*}
\langle \mathbf{\Gamma }_{jr}^{l}\mathbf{\Gamma }_{ki}^{r}-\mathbf{\Gamma }%
_{ir}^{l}\mathbf{\Gamma }_{kj}^{r}\rangle _{\alpha }=\langle \mathbf{\delta }%
_{r}^{l}\mathbf{\delta }_{i}^{r}\boldsymbol{\eth }_{\mathbf{\varrho }}%
\mathbf{g}_{j}\boldsymbol{\eth }_{\mathbf{\varrho }}\mathbf{g}_{k}-\mathbf{%
\delta }_{r}^{l}\mathbf{\delta }_{j}^{r}\boldsymbol{\eth }_{\mathbf{\varrho }%
}\mathbf{g}_{i}\boldsymbol{\eth }_{\mathbf{\varrho }}\mathbf{g}_{k}\rangle
_{\alpha }=
\end{equation*}%
\begin{equation*}
=2\langle (\mathbf{\delta }_{i}^{l}\boldsymbol{\eth }_{\mathbf{\varrho }}%
\mathbf{g}_{j}-\mathbf{\delta }_{j}^{l}\boldsymbol{\eth }_{\mathbf{\varrho }}%
\mathbf{g}_{i})\boldsymbol{\eth }_{\mathbf{\varrho }}\mathbf{g}_{k}\rangle
_{\alpha }\text{,}
\end{equation*}%
it follows that%
\begin{equation}
\langle \boldsymbol{\eth }_{ij}^{2}\mathbf{g}_{k}-\boldsymbol{\eth }_{ji}^{2}%
\mathbf{g}_{k}\rangle _{\alpha }=\langle (\boldsymbol{\eth }_{i}\mathbf{g}%
_{j}-\boldsymbol{\eth }_{j}\mathbf{g}_{i})\boldsymbol{\eth }_{\mathbf{%
\varrho }}\mathbf{g}_{k}\rangle _{\alpha }\text{,}  \label{80}
\end{equation}%
which confirms (\ref{75}). Furthermore,%
\begin{equation}
\langle \boldsymbol{\eth }_{i}\mathbf{\Gamma }_{kj}^{l}\rangle _{\alpha
}=\langle \boldsymbol{\eth }_{i}(\mathbf{\delta }_{j}^{l}\boldsymbol{\eth }_{%
\mathbf{\varrho }}\mathbf{g}_{k})\rangle _{\alpha }=\langle \boldsymbol{\eth 
}_{i}(\mathbf{\bar{g}}^{l}\boldsymbol{\eth }_{j}\mathbf{g}_{k})\rangle
_{\alpha }=\boldsymbol{\eth }_{i}(\boldsymbol{\eth }_{j}\mathbf{\delta }%
_{k}^{l}-\mathbf{g}_{k}\boldsymbol{\eth }_{j}\mathbf{\bar{g}}^{l})\rangle
_{\alpha }\text{ and}  \label{81}
\end{equation}%
\begin{equation*}
\langle \boldsymbol{\eth }_{j}\mathbf{\Gamma }_{ki}^{l}\rangle _{\alpha
}=\langle \boldsymbol{\eth }_{j}(\mathbf{\delta }_{i}^{l}\boldsymbol{\eth }_{%
\mathbf{\varrho }}\mathbf{g}_{k})\rangle _{\alpha }=\langle \boldsymbol{\eth 
}_{j}(\mathbf{\bar{g}}^{l}\boldsymbol{\eth }_{i}\mathbf{g}_{k})\rangle
_{\alpha }=\boldsymbol{\eth }_{j}(\boldsymbol{\eth }_{i}\mathbf{\delta }%
_{k}^{l}-\mathbf{g}_{k}\boldsymbol{\eth }_{i}\mathbf{\bar{g}}^{l})\rangle
_{\alpha }\text{.}
\end{equation*}%
Consequently,%
\begin{equation}
\langle \boldsymbol{\eth }_{i}\mathbf{\Gamma }_{kj}^{l}-\boldsymbol{\eth }%
_{j}\mathbf{\Gamma }_{ki}^{l}\rangle _{\alpha }=\langle \mathbf{t}_{j}^{l}%
\mathbf{t}_{ki}-\mathbf{t}_{i}^{l}\mathbf{t}_{kj}\rangle _{\alpha }+\langle 
\mathbf{\bar{g}}^{l}(\boldsymbol{\eth }_{ji}^{2}\mathbf{g}_{k}-\boldsymbol{%
\eth }_{ij}^{2}\mathbf{g}_{k})\rangle _{\alpha }\text{ and}  \label{82}
\end{equation}%
\begin{equation*}
\langle \boldsymbol{\eth }_{j}\mathbf{\Gamma }_{ki}^{l}-\boldsymbol{\eth }%
_{i}\mathbf{\Gamma }_{kj}^{l}+\mathbf{\Gamma }_{jr}^{l}\mathbf{\Gamma }%
_{ki}^{r}-\mathbf{\Gamma }_{ir}^{l}\mathbf{\Gamma }_{kj}^{r}\rangle _{\alpha
}=\langle \mathbf{t}_{j}^{l}\mathbf{t}_{ki}-\mathbf{t}_{i}^{l}\mathbf{t}%
_{kj}\rangle _{\alpha }\text{.}
\end{equation*}%
where $\langle \mathbf{t}_{j}^{l}\rangle _{\alpha }=\langle \boldsymbol{\eth 
}_{j}\mathbf{\bar{g}}^{l}\rangle _{\alpha }$ and $\langle \mathbf{t}%
_{ki}\rangle _{\alpha }=\langle \boldsymbol{\eth }_{i}\mathbf{g}_{k}\rangle
_{\alpha }$.

If $\langle \mathbf{R}_{kji}^{l}\rangle _{\alpha }=\langle \boldsymbol{\eth }%
_{j}\mathbf{\Gamma }_{ki}^{l}-\boldsymbol{\eth }_{i}\mathbf{\Gamma }%
_{kj}^{l}+\mathbf{\Gamma }_{jr}^{l}\mathbf{\Gamma }_{ki}^{r}-\mathbf{\Gamma }%
_{ir}^{l}\mathbf{\Gamma }_{kj}^{r}\rangle _{\alpha }$, then 
\begin{equation}
\langle \mathbf{R}_{kji}^{l}\rangle _{\alpha }=\langle \mathbf{t}_{j}^{l}%
\mathbf{t}_{ki}-\mathbf{t}_{i}^{l}\mathbf{t}_{kj}\rangle _{\alpha }=\langle 
\mathbf{g}_{j}\mathbf{g}_{i}(\boldsymbol{\eth }_{\mathbf{\varrho }}\mathbf{%
\bar{g}}^{l}\boldsymbol{\eth }_{\mathbf{\varrho }}\mathbf{g}_{k}-\boldsymbol{%
\eth }_{\mathbf{\varrho }}\mathbf{\bar{g}}^{l}\boldsymbol{\eth }_{\mathbf{%
\varrho }}\mathbf{g}_{k})\rangle _{\alpha }=\langle \mathbf{0}\rangle
_{\alpha }\text{ and}  \label{83}
\end{equation}%
\begin{equation*}
\langle \boldsymbol{\eth }_{i}\mathbf{\Gamma }_{kj}^{l}-\boldsymbol{\eth }%
_{j}\mathbf{\Gamma }_{ki}^{l}\rangle _{\alpha }=\langle \mathbf{\Gamma }%
_{jr}^{l}\mathbf{\Gamma }_{ki}^{r}-\mathbf{\Gamma }_{ir}^{l}\mathbf{\Gamma }%
_{kj}^{r}\rangle _{\alpha }\text{.}
\end{equation*}

Accordingly, the first identity in (\ref{83}) provides the vector form of
the \textit{Gauss} equation for \textit{Hermitian} manifolds within the
present geometric algebraic framework. Similarly,%
\begin{equation}
\langle \mathbf{0}\rangle _{\alpha }=\langle \mathbf{\bar{g}}_{i}\mathbf{g}%
_{j}\boldsymbol{\eth }_{\mathbf{\bar{\varrho}\varrho }}^{2}-\mathbf{g}_{j}%
\mathbf{\bar{g}}_{i}\boldsymbol{\eth }_{\mathbf{\bar{\varrho}\varrho }%
}^{2}\rangle _{\alpha }=\langle \boldsymbol{\bar{\eth }}_{i}\boldsymbol{\eth 
}_{j}-\boldsymbol{\eth }_{j}\boldsymbol{\bar{\eth }}_{i}-(\boldsymbol{\bar{%
\eth }}_{i}\mathbf{g}_{j}\boldsymbol{\eth }_{\mathbf{\varrho }}-\boldsymbol{%
\eth }_{j}\mathbf{\bar{g}}_{i}\boldsymbol{\eth }_{\mathbf{\bar{\varrho}}%
})\rangle _{\alpha }\text{ and}  \label{84}
\end{equation}%
\begin{equation*}
\langle \mathbf{g}^{ij}(\boldsymbol{\eth }_{i}\boldsymbol{\bar{\eth }}_{j}-%
\boldsymbol{\bar{\eth }}_{j}\boldsymbol{\eth }_{i})\rangle _{\alpha
}=\langle 2\boldsymbol{\eth }_{\mathbf{\bar{\varrho}}}\mathbf{\bar{g}}^{i}%
\boldsymbol{\eth }_{i}\rangle _{\alpha }\text{.}
\end{equation*}%
Furthermore,%
\begin{equation}
2\langle \boldsymbol{\eth }_{j}\boldsymbol{\bar{\eth }}_{i}-\boldsymbol{\bar{%
\eth }}_{i}\boldsymbol{\eth }_{j}\rangle _{\alpha }=\langle (\mathbf{g}_{j}%
\boldsymbol{\eth }_{\mathbf{\varrho }}\mathbf{\bar{\delta}}_{i}^{k}-\mathbf{%
\bar{g}}_{i}\boldsymbol{\eth }_{\mathbf{\bar{\varrho}}}\mathbf{\delta }%
_{j}^{k})\partial _{k}\rangle _{\alpha }=\langle (\boldsymbol{\eth }_{j}%
\mathbf{\bar{\delta}}_{i}^{k}-\boldsymbol{\bar{\eth }}_{i}\mathbf{\delta }%
_{j}^{k})\partial _{k}\rangle _{\alpha }\text{,}  \label{85}
\end{equation}%
since%
\begin{equation}
4\langle \boldsymbol{\eth }_{j}\boldsymbol{\bar{\eth }}_{i}\rangle _{\alpha
}=\langle \mathbf{\delta }_{j}^{l}\partial _{l}(\mathbf{\bar{\delta}}%
_{i}^{k}\partial _{k})\rangle _{\alpha }=\langle \mathbf{\delta }%
_{j}^{l}\partial _{l}\mathbf{\bar{\delta}}_{i}^{k}\partial _{k}+\mathbf{\bar{%
\delta}}_{i}^{l}\mathbf{\delta }_{j}^{k}\partial _{kl}^{2}\rangle _{\alpha }%
\text{ and}  \label{86}
\end{equation}%
\begin{equation*}
4\langle \boldsymbol{\bar{\eth }}_{i}\boldsymbol{\eth }_{j}\rangle _{\alpha
}=\langle \mathbf{\bar{\delta}}_{i}^{l}\partial _{l}(\mathbf{\delta }%
_{j}^{k}\partial _{k})\rangle _{\alpha }=\langle \mathbf{\bar{\delta}}%
_{i}^{l}\partial _{l}\mathbf{\delta }_{j}^{k}\partial _{k}\rangle _{\alpha }+%
\mathbf{\bar{\delta}}_{i}^{l}\mathbf{\delta }_{j}^{k}\partial
_{kl}^{2}\rangle _{\alpha }\text{.}
\end{equation*}

\begin{remark}
The commutator identities $($\ref{77}$)$ and $($\ref{85}$)$ show that the
component \textit{Wirtinger} operators form a non-commutative differential
operator algebra, and that the associated vector \textit{Christoffel}
symbols generally satisfy $\mathbf{\Gamma }_{ij}^{k}\neq \mathbf{\Gamma }%
_{ji}^{k}$, unlike the classical \textit{Christoffel} symbols of \textit{%
Riemannian} geometry.
\end{remark}

\section{Conclusion}

A commutative geometric framework for the 3$\mathbb{D}$ field of complex
vectors $\mathbf{V}_{\mathbb{C}}^{3}$ has been developed by inducing its
algebraic structure from the corresponding 2$\mathbb{D}$ component fields $%
\left\langle \mathbf{V}_{\mathbb{C}}\right\rangle _{\alpha }$. In this
construction, a complex vector is represented by an ordered pair consisting
of a real vector and an imaginary vector, induced from the ordered pairs of
the component fields, in complete analogy with the ordered-pair
representation of complex numbers. This approach resolves the algebraic
difficulties associated with extending the notion of complex conjugation and
vector inversion to three dimensions, while preserving a consistent
commutative geometric product. The integral identity $vt\int_{\partial
\Omega ^{+}}^{\circlearrowleft }\mathbf{\omega }=vt\int_{\Omega ^{+}}\mathtt{%
d}\mathbf{\omega }$ providing a common formulation of the fundamental
theorem of integral calculus in the 2$\mathbb{D}$ and 3$\mathbb{D}$ settings.

The developed formalism naturally extends to \textit{Hermitian} manifolds
embedded in the 3$\mathbb{D}$ field of complex vectors $\mathbf{V}_{\mathbb{C%
}}^{3}$. The induced \textit{Wirtinger} operators generate the corresponding
vector differential geometry, including vector analogues of the \textit{%
Christoffel} symbols, the \textit{Riemann} curvature tensor, and the \textit{%
Codazzi--Mainardi} and \textit{Gauss} equations. In contrast to the
classical \textit{Levi--Civita} connection of \textit{Riemannian} geometry,
the associated vector \textit{Christoffel} symbols are generally not
symmetric with respect to their lower indices, reflecting the intrinsic
non-commutativity of the induced differential operator algebra.

The results noted above demonstrate that the proposed commutative geometric
product provides a unified algebraic framework from which differential
operators and integral identities, as well as differential-geometric
structures arise in a natural and coherent manner. The proposed 3$\mathbb{D}$
field of complex vectors $\mathbf{V}_{\mathbb{C}}^{3}$ therefore offers an
alternative setting for the study of vector fields, extending the
correspondence between complex analysis and vector calculus to a broader
geometric context.

\section{Statements and Declarations}

The author declares that no funds, grants, or other support were received
during the preparation of this manuscript.

The author has no relevant financial or non-financial interests to disclose.

\end{document}